\documentclass[10pt,reqno]{amsart}

\usepackage{amscd}
\usepackage{amsfonts}
\usepackage{amsmath}
\usepackage{amssymb}
\usepackage{amsthm}
\usepackage{enumitem}
\usepackage{float}
\usepackage{mathtools}
\usepackage{mathrsfs}
\usepackage{needspace}
\usepackage{setspace}
\usepackage{tikz}
\usepackage{version}

\usepackage[colorlinks=true, linkcolor=blue, urlcolor=blue, citecolor=blue,%
pagebackref=true]{hyperref}

\onehalfspace

\theoremstyle{plain}
\newtheorem{lemma}{Lemma}[section]
  \newtheorem{proposition}[lemma]{Proposition}
  \newtheorem{corollary}[lemma]{Corollary}
  \newtheorem{theorem}[lemma]{Theorem}
  \newtheorem*{theorem*}{Theorem}

  \newtheorem*{fact*}{Fact}
  \newtheorem*{claim*}{Claim}

\theoremstyle{definition}
  \newtheorem{definition}[lemma]{Definition}
  
    \newtheorem{remark}[lemma]{Remark}
  \newtheorem{example}[lemma]{Example}

\theoremstyle{remark}

\setenumerate{label=\textup{(}\roman*\textup{)}}

\makeatletter
\newsavebox{\@brx}
\newcommand{\llangle}[1][]{\savebox{\@brx}{\(\m@th{#1\langle}\)}%
  \mathopen{\copy\@brx\kern-0.5\wd\@brx\usebox{\@brx}}}
\newcommand{\rrangle}[1][]{\savebox{\@brx}{\(\m@th{#1\rangle}\)}%
  \mathclose{\copy\@brx\kern-0.5\wd\@brx\usebox{\@brx}}}
\makeatother

\let\pol\l
\usepackage[alphabetic,abbrev,msc-links,nobysame]{amsrefs}

\DeclareMathOperator{\Homeo}{Homeo}
\DeclareMathOperator{\RP}{RP}

\title[Representation as nilmanifolds]{The structure theory of Nilspaces II: Representation as nilmanifolds}

\author{Yonatan Gutman, Freddie Manners and P\'{e}ter P. Varj\'{u}}


\address{Yonatan Gutman, Institute of Mathematics, Polish Academy of Sciences,
ul. \'{S}niadeckich~8, 00-656 Warszawa, Poland.}
\email{y.gutman@impan.pl}

\address{Freddie Manners, Mathematical Institute, Radcliffe Observatory Quarter, Woodstock Road, Oxford OX2 6GG}
\email{Frederick.Manners@maths.ox.ac.uk}

\address{P\'{e}ter P. Varj\'{u}, Centre for Mathematical Sciences,
Wilberforce Road, Cambridge CB3 0WA,
UK}
\email{pv270@dpmms.cam.ac.uk}
\date{\today}

\thanks{YG was partially supported by the ERC Grant \emph{Approximate Algebraic Structures and Applications} and the NCN (National Science Center, Poland) grant 2016/22/E/ST1/00448.
PPV was supported by the Royal Society.}

\begin{document}

\newcommand*{\singlesquare}[4]{
  \draw[thin,black] (0,0) node[below left] {#1} -- (1, 0) node[below right] {#2} -- (1, 1) node[above right] {#4} -- (0, 1) node[above left] {#3} -- (0,0);
}

\newcommand*{\doublesquare}[7]{
  \begin{scope}[xscale=#7]
    \draw[thin,black] (0,0) node[below left] {#4} -- (1, 0) node[below] {#5} -- (2, 0) node[below right] {#6} -- (2, 1) node[above right] {#3} -- (1, 1) node[above] {#2} -- (0, 1) node[above left] {#1} -- (0,0);
    \draw[thin,black] (1, 0) -- (1, 1);
  \end{scope}
}

\newcommand*{\threecube}[9]{
  \begin{scope}[x={(#9, 0)}, y={(0, 1)}, z={(0.352, 0.317)}, scale=2]
    \draw (0,0,0) node[below left] {#1} -- (0,0,1) node[below right] {#3} -- (0,1,1) node[above] {#7} -- (0,1,0) node[above left] {#5} -- cycle;
    \draw (0,0,0) -- (1,0,0) node[below right] {#2} -- (1,0,1) node[right] {#4} -- (0,0,1) -- cycle;
    \draw (1,1,1) node[above right] {#8} -- (0,1,1) -- (0,0,1) -- (1,0,1) -- cycle;

    \draw (0,0,0) -- (1,0,0) -- (1,1,0) node [above left] {#6} -- (0,1,0) -- cycle;
    \draw (1,1,1) -- (1,1,0) -- (1,0,0) -- (1,0,1) -- cycle;
    \draw (1,1,1) -- (0,1,1) -- (0,1,0) -- (1,1,0) -- cycle;
  \end{scope}
}

\newcommand*{\inlinetikz}[1]{
  \begin{figure}[H]
    \centering
    \begin{tikzpicture}
      #1
    \end{tikzpicture}
  \end{figure} \noindent
}
\usetikzlibrary{patterns}

\newcommand{\eps}[0]{\varepsilon}

\newcommand{\AAA}[0]{\mathbb{A}}
\newcommand{\CC}[0]{\mathbb{C}}
\newcommand{\EE}[0]{\mathbb{E}}
\newcommand{\FF}[0]{\mathbb{F}}
\newcommand{\NN}[0]{\mathbb{N}}
\newcommand{\PP}[0]{\mathbb{P}}
\newcommand{\QQ}[0]{\mathbb{Q}}
\newcommand{\RR}[0]{\mathbb{R}}
\newcommand{\TT}[0]{\mathbb{T}}
\newcommand{\ZZ}[0]{\mathbb{Z}}

\newcommand{\cA}[0]{\mathcal{A}}
\newcommand{\cB}[0]{\mathcal{B}}
\newcommand{\cC}[0]{\mathcal{C}}
\newcommand{\cD}[0]{\mathcal{D}}
\newcommand{\cE}[0]{\mathcal{E}}
\newcommand{\cF}[0]{\mathcal{F}}
\newcommand{\cH}[0]{\mathcal{H}}
\newcommand{\cG}[0]{\mathcal{G}}
\newcommand{\cK}[0]{\mathcal{K}}
\newcommand{\cM}[0]{\mathcal{M}}
\newcommand{\cN}[0]{\mathcal{N}}
\newcommand{\cP}[0]{\mathcal{P}}
\newcommand{\cR}[0]{\mathcal{S}}
\newcommand{\cS}[0]{\mathcal{S}}
\newcommand{\cT}[0]{\mathcal{S}}
\newcommand{\cU}[0]{\mathcal{U}}
\newcommand{\cW}[0]{\mathcal{W}}
\newcommand{\cX}[0]{\mathcal{X}}
\newcommand{\cY}[0]{\mathcal{Y}}
\newcommand{\cZ}[0]{\mathcal{Z}}

\newcommand{\fg}[0]{\mathfrak{g}}
\newcommand{\fk}[0]{\mathfrak{k}}
\newcommand{\fZ}[0]{\mathfrak{Z}}

\newcommand{\bmu}[0]{\boldsymbol\mu}

\newcommand{\AUT}[0]{\mathbf{Aut}}
\newcommand{\Aut}[0]{\operatorname{Aut}}
\newcommand{\Frob}[0]{\operatorname{Frob}}
\newcommand{\GI}[0]{\operatorname{GI}}
\newcommand{\HK}[0]{\operatorname{HK}}
\newcommand{\HOM}[0]{\mathbf{Hom}}
\newcommand{\Hom}[0]{\operatorname{Hom}}
\newcommand{\Ind}[0]{\operatorname{Ind}}
\newcommand{\Lip}[0]{\operatorname{Lip}}
\newcommand{\LHS}[0]{\operatorname{LHS}}
\newcommand{\RHS}[0]{\operatorname{RHS}}
\newcommand{\Sub}[0]{\operatorname{Sub}}
\newcommand{\id}[0]{\operatorname{id}}
\newcommand{\image}[0]{\operatorname{Im}}
\newcommand{\poly}[0]{\operatorname{poly}}
\newcommand{\trace}[0]{\operatorname{Tr}}
\newcommand{\sig}[0]{\ensuremath{\tilde{\cS}}}
\newcommand{\psig}[0]{\ensuremath{\cP\tilde{\cS}}}
\newcommand{\metap}[0]{\operatorname{Mp}}
\newcommand{\symp}[0]{\operatorname{Sp}}
\newcommand{\dist}[0]{\operatorname{dist}}
\newcommand{\stab}[0]{\operatorname{Stab}}
\newcommand{\HCF}[0]{\operatorname{hcf}}
\newcommand{\LCM}[0]{\operatorname{lcm}}
\newcommand{\SL}[0]{\operatorname{SL}}
\newcommand{\GL}[0]{\operatorname{GL}}
\newcommand{\rk}[0]{\operatorname{rk}}
\newcommand{\sgn}[0]{\operatorname{sgn}}
\newcommand{\uag}[0]{\operatorname{UAG}}
\newcommand{\freiman}[0]{Fre\u{\i}man}
\newcommand{\tf}[0]{\operatorname{tf}}
\newcommand{\ev}[0]{\operatorname{ev}}

\newcommand{\Conv}[0]{\mathop{\scalebox{1.5}{\raisebox{-0.2ex}{$\ast$}}}}
\newcommand{\bs}[0]{\backslash}

\newcommand{\heis}[3]{ \left(\begin{smallmatrix} 1 & \hfill #1 & \hfill #3 \\ 0 & \hfill 1 & \hfill #2 \\ 0 & \hfill 0 & \hfill 1 \end{smallmatrix}\right)  }

\newcommand{\uppar}[1]{\textup{(}#1\textup{)}}

\begin{abstract}
This paper forms the second part of a series by the authors \cites{GMV1,GMV3} concerning the structure theory of \emph{nilspaces} of Antol\'\i n Camarena and Szegedy. A nilspace is a compact space $X$ together with closed collections of \emph{cubes} $C_n(X)\subseteq X^{2^n}$, $n=1,2,\ldots$ satisfying some natural axioms. From these axioms it follows that a nilspace can be built as a finite tower of extensions where each of the successive fibers is a compact abelian group.

Our main result is a new proof of a result due to Antol\'\i n Camarena and Szegedy \cite{CS12}, stating that if each of these groups is a torus then $X$ is isomorphic (in a strong sense) to a nilmanifold $G/\Gamma$. We also extend the theorem to a setting where the nilspace arises from a dynamical system $(X,T)$. These theorems are a key stepping stone towards the general structure theorem in \cite{GMV3} (which again closely resembles the main theorem of \cite{CS12}).

The main technical tool, enabling us to  deduce algebraic information from topological data, consists of existence and uniqueness results for solutions of certain natural functional equations, again modelled on the theory in \cite{CS12}.
\end{abstract}

\maketitle{}

\tableofcontents{}

\section{Introduction}

This paper forms part of a series by the authors \cites{GMV1,GMV3} concerning the structure theory of \emph{nilspaces}.  Much of this is concerned with the approach of Szegedy \cite{S12} and Antol\'\i n Camarena and Szegedy \cite{CS12} to the inverse theorem for the Gowers norms, as well as with relations to dynamics and in particular work of Host and Kra \cite{HK05} and Host, Kra and Maass \cite{HKM10}.

The paper \cite{GMV1} contains an extensive introduction to this project from the viewpoint of higher order Fourier analysis and the inverse theorem for the Gowers norms.  Similarly, \cite{GMV3} introduces the project from a dynamical perspective.  We will not repeat the bulk of these introductions here, nor offer much motivation for the definition or study of nilspaces and related constructs, but instead refer the reader to these companion works.

 Nilspaces originate in work of Host and Kra \cite{HK08}, where these objects appeared under the name of  ``parallelepiped structures''.  The study of these objects was furthered by Antol\'\i n Camarena and Szegedy \cite{CS12}, who in the same work formulated a strong structure theorem for nilspaces, subject to certain further hypotheses.

The papers of Candela \cites{Can1,Can2} expand on \cite{CS12}, providing more detailed proofs. He also includes
several additional results implicit in \cite{CS12}, particularly about continuous systems of measures.

The goals of this work are as follows.

\begin{itemize}
  \item We prove a structure theorem for nilspaces with certain additional topological assumptions, which allow us to deduce that they are isomorphic (in a suitable sense) to nilmanifolds $G/\Gamma$.  This is a key stepping stone towards the general structure theorem.

  \item Along the way, we prove some rather technical results concerning ``cocycles'' on cubespaces (closely related to \cite{CS12}*{Section 3.6}).  The resulting primitive ``cohomology'' theory is a powerful tool for deducing algebraic information from topological data, and will be invaluable both here and elsewhere in the project \cite{GMV3}.
\end{itemize}

The structural result we will prove here roughly has the following flavour: if a (compact, ergodic) nilspace $X = (X, C^n(X))$, satisfying some connectivity hypotheses, has any chance of being a nilmanifold topologically -- for instance, it had better be a topological manifold, i.e.~locally homeomorphic to a subset of $\RR^n$ -- then it is a nilmanifold.

One possible formulation of the statement is the following theorem.
(For the definition of nilspaces we refer the reader to \cite{GMV1}*{Section 3.1} 
or \cite{GMV3}*{Section 1.3}, and for that of Host--Kra cubes, 
see \cite{GMV1}*{Section 2 and Appendix A} or \cite{GMV3}*{Section 1.4}). 

\begin{theorem}
  \label{main-thm-simple}
  Let $X = (X, C^n(X))$ be a \uppar{compact, ergodic}\footnote{The assumption that a nilspace be compact and ergodic is in force universally, and so it usually appears in parentheses to downplay its particular significance in any given context. However, these conditions are certainly not optional.} nilspace of degree $s$.  Suppose $X$ is locally connected and of finite Lebesgue covering dimension, and further that all the spaces $C^n(X)$ are connected.

  Then $X$ is isomorphic to a nilmanifold $G/\Gamma$.  That is, there exists a filtered connected Lie group $G_\bullet$, a discrete co-compact subgroup $\Gamma$ of $G$, and a homeomorphism $\phi \colon X \leftrightarrow G/\Gamma$ that identifies the cubes $C^k(X)$ with the Host--Kra cubes $\HK^k(G_\bullet)/\Gamma$.
\end{theorem}

 We note that the conditions on local connectedness and finite dimensionality hold in particular if $X$ is a topological manifold.

 The topological conditions in the above theorem can be replaced by other sets of assumptions involving the so-called structure groups of the nilspace $X$. Such variants will be formulated below, including the one proved by Antol\'\i n Camarena and Szegedy \cite{CS12}*{Theorem 7}. In fact, we only prove one of these variants in this paper and the proof of Theorem \ref{main-thm-simple} is completed in \cite{GMV3}*{Theorem A.1} 
 using the general structure theory of nilspaces.

Both the overall structure of these arguments, and a good part of the fine detail, are modelled closely on the work of Antol\'\i n Camarena and Szegedy, and we will attempt to make this dependence explicit wherever possible.  However, in some aspects we have deviated from their approach.  Our reasons for doing so are some mixture of the following.
\begin{itemize}
  \item In places, we obtain slightly stronger results, which will be useful especially in applications to dynamics.  Obtaining these can require a modified approach.
  \item Our arguments stay entirely within the topological category, avoiding reference to objects that are merely measurable rather than continuous.  This eliminates some subtleties encountered in \cite{CS12} concerning the replacement of measurable objects by continuous ones; albeit arguably at the expense of introducing other subtleties in different places.  In particular this strategy requires a different treatment of the part of the argument covered in Section \ref{sec-main-details}.  Similarly, for some results in Section \ref{sec:cocycle}, we will have to prove strengthened versions, which are not required here but are needed to keep the arguments in \cite{GMV3} in the topological category.
  \item In some cases, we find our alternatives simpler, easier to understand
  or to yield a fuller understanding of the methods and structures involved.
\end{itemize}

\subsection{Structure of the paper}

As we have said, we will not repeat in detail the relevant definitions of cubespaces, nilspaces, ergodicity and so on, that appear in \cite{GMV1}*{Section 3.1}, 
referring the reader to that paper.  Similarly, we will assume some familiarity with the crucial subject of the \emph{structure groups} $A_k(X)$ of a nilspace, and with the ``weak structure theorem'' of \cite{GMV1}*{Section 5.1}. 

Alternatively, a reader mainly interested in dynamics may read \cite{GMV3}*{Sections 1.3--1.6}. 
The paper  \cite{GMV3}  introduces the same notions and their properties motivated from a dynamical viewpoint.
However, the proofs are available only in \cite{GMV1}.

We will recall only some of the more specialized definitions relating to Theorem \ref{main-thm-simple} and its generalizations, in Section \ref{sec-statements} below.  The same section will introduce formally the group $\Aut_k(X)$ of \emph{$k$-translations} of a nilspace, which will play a key role in the arguments.  Section \ref{sec-statements} concludes with a discussion of the various variations on the statement of Theorem \ref{main-thm-simple}.

Section \ref{sec:main-outline} explains the high-level steps of the proof of the structural results.
This reduces the structural result to a statement (see Proposition \ref{lem:surj})
that a nilmanifold $X$ has ``enough automorphisms'' in a certain technical sense.

We establish this technical statement in Section \ref{sec-main-details}, conditional on a
``cohomological'' theorem (see Theorem \ref{thm:cocycle-special}) that will arise fairly naturally in the course of the proof.

We introduce the cocycle theory, and give a proof of this remaining theorem, in Section \ref{sec:cocycle}.

\subsection{Acknowledgments}

First and foremost we owe gratitude to Bernard Host who introduced us to the subject and to Omar Antol\'\i n Camarena and Bal\'azs Szegedy
whose groundbreaking work \cite{CS12} was a constant inspiration
for us.

We would like to thank Emmanuel Breuillard, J\'er\^ome Buzzi,
Yves de Cornulier, Sylvain Crovisier, Eli Glasner, Ben Green, Bernard Host, Micha\pol{} Rams, Bal\'azs Szegedy, Anatoly Vershik and Benjamin
Weiss for helpful discussions. We are grateful to Pablo Candela and Bryna Kra for a careful reading of a preliminary version.  We are grateful to Jacob Rasmussen for suggesting the reference \cite{S51}.

We are grateful to the referee for her or his careful reading of our paper and for
her or his many helpful comments, which greatly improved the presentation of the paper.

\section{Definitions and statements}
\label{sec-statements}

The flavour of all our structural statements is to find conditions one can impose on a cubespace $X = (X, C^n(X))$ that are sufficient to ensure that it is actually a nilmanifold $G / \Gamma$.

The various ``algebraic'' constraints we impose are discussed at length in \cite{GMV1}.  In short, it is fairly hopeless to ask for a rigid structure theorem unless we insist that $X$ is an ergodic nilspace.

However, we also require some topological input (beyond the standard assumption that $X$ is compact).  For instance, a nilmanifold is certainly a smooth real manifold, so we need to rule out examples such as the solenoid\footnote{Here $\ZZ_2$ denotes the $2$-adic integers and $\ZZ$ is embedded diagonally in the product.} $(\RR \times \ZZ_2) / \ZZ$, which is a compact abelian group and hence a nilspace of degree $1$, but not a manifold.

In the previous section, we formulated Theorem \ref{main-thm-simple} under topological conditions on the space $X$ and the space of cubes $C^n(X)$. We will now state a similar result in which these conditions are replaced by conditions on the \emph{structure groups} of $X$.  We recall that these are a sequence of compact abelian groups $A_k(X)$, defined canonically in terms of the cubestructure on $X$.  The key topological hypothesis we can impose is that these are Lie groups.  \footnote{Note we do not assume a Lie group is connected.  So, a compact abelian Lie group is precisely of the form $(\RR/\ZZ)^d \times K$ where $K$ is a (discrete) finite group.}

These formulations turn out to fit more naturally in the theory, and we prove only these statements in this paper.

Unfortunately the condition that $A_k(X)$ are Lie still does not suffice.  Our method also requires some fairly strong assumptions on connectivity, to which we now turn.

\subsection{Connectivity hypotheses}

The simplest connectivity hypothesis is just that $X$ itself is connected as a topological space. It turns out that this condition becomes inadequate fairly quickly, because it fails to establish connectedness ``on all levels'' that is required to explore the structure of $X$ inductively.

Sticking with the idea of imposing conditions on the structure groups $A_k(X)$, one can insist that these groups are connected for all $k$.  Since we have already assumed that they are Lie, this is equivalent to asking that they be tori $(\RR/\ZZ)^d$ for some integer $d$.

\begin{definition}
  We say a (compact, ergodic) nilspace $X$ is \emph{toral} if for each $k \ge 1$ the structure group\footnote{For a definition see \cite{GMV1}*{Theorem 5.4}.} $A_k(X)$ is isomorphic to a torus $(\RR/\ZZ)^d$ for some integer $d$ (which depends on $X$ and $k$).
\end{definition}

This is the approach taken in \cite{CS12}.

\begin{theorem}[\cite{CS12}*{Theorem 7}]\label{thm:main-toral}
  Let $X = (X, C^n(X))$ be a \uppar{compact, ergodic} toral nilspace of degree $s$. Then $X$ is isomorphic to a nilmanifold $G/\Gamma$ in the sense of Theorem \ref{main-thm-simple}.
\end{theorem}

For the purposes of our core structural result, Theorem \ref{thm:main-sc} below, we will work with a slightly different topological condition, this time relating to the spaces of cubes $C^k(X)$.  This is needed for compatibility with our statement of the final structure theorem for general nilspaces  (\cite{GMV1}*{Theorem 4.1}).

\begin{definition}
  We say a nilspace $X$ is \emph{strongly connected} if the topological space $C^k(X)$ is connected for all $k \ge 0$. (In particular, $X = C^0(X)$ itself is connected.)
\end{definition}

It turns out that these two connectivity conditions are equivalent.
While it is easy to see that a toral nilspace is strongly connected
(see Proposition \ref{prop-toral-sc}), the converse appears much more difficult.
In fact, the only proof we are aware of is based on the full force of our
structure theorem.
Indeed, Theorem \ref{thm:main-sc} below implies that a strongly connected nilspace with Lie structure
groups is isomorphic to a Host--Kra nilspace of a connected nilpotent Lie group endowed with a filtration
of connected subgroups, and hence is a toral nilspace \emph{a fortiori}.

\begin{proposition}
  \label{prop-toral-sc}
  A toral \uppar{compact, ergodic} nilspace is strongly connected.
\end{proposition}
\begin{proof}
  By the weak structure theorem (\cite{GMV1}*{Theorem 5.4}) 
  the space $C^n(X)$ is expressible as a tower of extensions (more correctly, principal bundles)
  \[
    C^n(X) = C^n(X_s) \to C^n(X_{s-1}) \to \dots \to C^n(X_1) \to \{\ast\}
  \]
  where each fiber is the compact abelian group $C^n(\cD_k(A_k(X)))$, which in turn is isomorphic to $A_k(X)^d$ for some $d \ge 0$, and in particular is connected.
  Since a tower of extensions by connected fibers is connected, we deduce that $C^n(X)$ is connected.
\end{proof}

In fact our techniques allow us to say something even in the absence of any such connectivity hypothesis, and we will make statements in this setting as well (see Theorem \ref{weaker-main-thm} and Corollary \ref{cor-weaker-main-thm} in the sequel). These statements are especially interesting from the viewpoint of topological dynamics.

\subsection{Some notation}

We recall some miscellaneous notation from \cite{GMV1}.  Given two configurations $c, c' \colon \{0,1\}^n \to X$, we denote by $[c, c']$ the ``concatenated'' configuration
\begin{align*}
  [c,c'] \colon \{0,1\}^{n+1} &\to X  \\
\omega &\mapsto \begin{cases} c(\omega_1,\dots,\omega_n) &\colon \omega_{n+1} = 0 \\  c'(\omega_1,\dots,\omega_n) &\colon \omega_{n+1} = 1 \end{cases} \ .
\end{align*}
We will not be too concerned about which coordinate in $\{1, \dots, n+1\}$ is the preferred one along which the concatenation occurs; i.e.~some other coordinate may play the role of $(n+1)$ in this expression, and we will still denote the final configuration by $[c, c']$.

We also use $\square^n(x)$ to denote the constant configuration $\{0,1\}^n \to X$ sending every coordinate to $x$.  Similarly, the notation $\llcorner^n(x; y)$ denotes the configuration
\begin{align*}
  \llcorner^n(x; y) \colon \{0,1\}^n &\to X \\
  \omega &\mapsto \begin{cases} y &\colon \omega = \vec{1} \\ x &\colon \omega \ne \vec{1} \end{cases} \ .
\end{align*}

More generally, we use the notation $\square^n(c)$ where $c\in X^{\{0,1\}^k}$ to denote the $(n+k)$-configuration $(\omega_1,\ldots,\omega_{n+k})\mapsto c(\omega_1,\ldots,\omega_k)$. We use the notation $\llcorner^n(c_1; c_2)$ in a similar manner for $c_1,c_2\in X^{\{0,1\}^k}$.

We reserve the right to mix and abuse notation freely, as in $\llcorner^2(\square^3(x); \square^3(y))$.  Hopefully the meaning will always be clear.

Finally, we use the standard notation $X \lesssim_{a,b,\dots} Y$ to denote that $X \le C(a,b,\dots) Y$ for some constant $C$ depending only on the variables $(a,b,\dots)$.

\subsection{Automorphisms and translations}

The heart of the problem of identifying a suitable nilspace $X$ with a nilmanifold $G / \Gamma$, is in recovering the group $G$ and its group law.

The approach to doing this taken in \cite{CS12}, which we follow, is to consider a suitable group of automorphisms of $X$.  Since automorphism groups are clearly groups, we will be in good shape if we can take $G$ to be this group, and then identify $X$ with a suitable quotient of $G$.

However, we want to describe not just $X$ as a topological space, but also its cube structure.  The cubes on a nilmanifold $G / \Gamma$ are given by the Host--Kra construction, which requires the additional data of a \emph{filtration} on the group $G$.  (For a detailed account of this theory, again see \cite{GMV1}*{Section 2 and Appendix A}.) 
So, we need to find not just a suitable group of automorphisms of $X$, but also a filtration on that group.

Fortunately, there are natural definitions of all of these objects, which we now describe.  The automorphism group itself is straightforward.

\begin{definition}
  Let $X$ be a compact cubespace.  We write $\Aut(X)$ for the group of all automorphisms of $X$ in the category of cubespaces; that is, the collection of homeomorphisms $\phi \colon X \to X$ such that for any configuration $c \colon \{0,1\}^k \to X$, $c \in C^k(X)$ if and only if $\phi(c) \in C^k(X)$.

  We endow $\Aut(X)$ with the usual supremum metric
 \[
 d(\phi,\psi) = \sup_{x \in X} d(\phi(x),\psi(x))=\|\psi\circ\phi^{-1}\|,
 \]
 where we used the notation:
  \[
    \|\phi\| = \sup_{x \in X} d(x, \phi(x)),
  \]
  which generate the usual compact-open topology.
\end{definition}
\begin{remark}
  \label{rem:second-countable-auts}
  It is straightforward to verify that $\Aut(X)$ is a closed subgroup of the group of homeomorphisms of $X$, and hence completely metrizable.\footnote{Note that we do not claim that the supremum metric itself on $\Aut(X)$ is complete; merely that $\Aut(X)$ is complete with respect to some metric which generates the same topology, e.g., $d'(\phi,\psi) = d(\phi,\psi) +d(\phi^{-1},\psi^{-1})$ (\cite{BK96}*{Corollary 1.2.2}).
The choice of the metric is not important, but we stick with the supremum metric for convenience of notation.}
It follows from a standard theorem (see \cite{BK96}*{Example 1.3(v)}), that $\Aut(X)$ is also separable, or equivalently second countable.
\end{remark}

We now consider the filtration.

\begin{definition}
  \label{def:k-translation}
  Again let $X$ be a compact cubespace, and fix $k \ge 0$. Given $\phi \in \Aut(X)$ and a face $F$ of $\{0,1\}^n$, let $[\phi]_{F}$ denote the element of $\Aut(X)^{\{0,1\}^n}$ given by
  \[
    \omega \mapsto \begin{cases} \phi &\colon \omega \in F \\ \id & \colon \omega \notin F \end{cases}
  \]
  We write $\Aut_k(X)$ for the collection of $\phi \in \Aut(X)$ with the following additional property.  For any integer $n \ge k$, any face $F$ of $\{0,1\}^n$ of codimension $k$, and any $c \in C^n(X)$, the configuration $[\phi]_{F}. c$
  \begin{align*}
    \{0,1\}^n &\to X \\
    \omega &\mapsto \begin{cases} \phi(c(\omega)) &\colon \omega \in F \\ c(\omega) &\colon \omega \notin F \end{cases}
  \end{align*}
  is in $C^n(X)$.

  We refer to the elements of $\Aut_k(X)$ as \emph{$k$-translations}.  Clearly $\Aut_0(X) = \Aut(X)$.
\end{definition}

The notion of translations originate from the work of Host and Kra \cite{HK08}*{Definition 6}
and they play a pivotal role in the programme of
Antol\'\i n Camarena and Szegedy \cite{CS12}, which we discuss now.

The motivation for this definition is perhaps not entirely clear.%
\footnote{
  Below is one way to arrive at this definition.  We stress that the following discussion is purely for motivation, and is not logically necessary for the argument (although some ideas discussed now will be relevant later; see for instance Proposition \ref{prop:evaluation-is-morphism}).

  We define a \emph{cubegroup} to be an object $G$ that is both a topological group and a cubespace, with the added requirement that $C^k(G)$ is a (closed) subgroup of $G^{\{0,1\}^k}$ under pointwise operations, for all $k$.

  It turns out that $\Aut(X)$ is a natural example of a cubegroup: that is, there is a canonical categorial notion of when $2^k$ automorphisms of $X$ form a $k$-cube, and this notion is closed under pointwise composition.  The cubespace structure can be described informally as the largest possible one such that $\omega \mapsto \phi_\omega(c(\omega))$ is a cube of $X$ for every $(\omega \mapsto \phi_\omega) \in C^k(\Aut(X))$ and every $c \in C^k(X)$, i.e.~such that pointwise action of a cube on $X$ sends cubes to cubes.  More precisely, one should define an element $\psi \colon \{0,1\}^k \to \Aut(X)$ to be in $C^k(\Aut(X))$ if and only if $\square^\ell(\psi)(c) \in C^{k+\ell}(X)$ for every $\ell \ge 0$ and $c \in C^{k+\ell}(X)$, and then one can check that this defines a cubegroup structure on $\Aut(X)$.

  It is a fact (which we will not prove, because we do not need it anywhere in the paper) that -- in complete generality -- all cubegroups arise from the Host--Kra construction applied to a filtered group.  That is, given a cubegroup $G$, there is a unique filtration $G_\bullet$ on $G$  such that $C^k(G) = \HK^k(G_\bullet)$ for all $k$, and so specifying a cubegroup structure on $G$ is equivalent to specifying a filtration.  Under this correspondence, the filtration $\Aut_k(X)$ from Definition \ref{def:k-translation} is precisely the one giving rise to the cubegroup structure on $\Aut(X)$ discussed above.

  In the interests of concreteness and simplicity we will suppress explicit discussion of the cube structure on $\Aut(X)$ in what follows (referring equivalently instead to $\HK^k(\Aut(X)_\bullet)$) and phrase everything in terms of filtrations, taking Definition \ref{def:k-translation} as the logical starting point.
}
We will illustrate it somewhat with some facts and examples.

\begin{proposition}
  \label{translation-prefiltration}
  Let $X$ be a compact cubespace.  The groups $(\Aut_k(X))_{k \ge 0}$ form a \emph{filtration}. That is, $\Aut_k(X)$ are a decreasing sequence of closed subgroups of $\Aut(X)$, and for any $i, j$ we have have the commutator inclusion $[\Aut_i(X), \Aut_j(X)] \subseteq \Aut_{i+j}(X)$.
\end{proposition}

\begin{proof}
  That $\Aut_k(X)$ forms a subgroup is clear, and it is similarly straightforward to argue that they are closed in $\Aut(X)$ (since $C^n(X)$ is closed in $X^{\{0,1\}^n}$ for all $n$).  Similarly, by modifying the same cube twice on an adjacent pair of faces, it is easy to see that a $(k+1)$-translation is also a $k$-translation for all $k$.

  For the commutator result, take any $\phi \in \Aut_i(X)$, $\psi \in \Aut_j(X)$ and $c \in \{0,1\}^n$ for some $n \ge i + j$, and let $F$ be a face of $\{0,1\}^n$ of codimension $(i+j)$.  Pick faces $F_1, F_2 \subseteq \{0,1\}^n$ of codimensions $i, j$ respectively such that $F_1 \cap F_2 = F$ and note that
  \[
    [\phi]_{F_1} [\psi]_{F_2}  [\phi]_{F_1}^{-1} [\psi]_{F_2}^{-1} = [\phi \psi \phi^{-1} \psi^{-1}]_F
  \]
  (with operations applied pointwise).  But clearly the configuration
  \[
    [\phi]_{F_1} [\psi]_{F_2}  [\phi]_{F_1}^{-1} [\psi]_{F_2}^{-1} . c
  \]
  is in $C^n(X)$ by hypothesis on $\phi$ and $\psi$, as required.
\end{proof}

\begin{proposition}
  \label{prop-translations-terminate}
  Suppose $X$ is a compact cubespace with $k$-uniqueness.  Then $\Aut_k(X) = \{\id\}$.
\end{proposition}
\begin{proof}
  For any $x \in X$ and $\phi \in \Aut_k(X)$, let $F = \{\vec{1}\} \subseteq \{0,1\}^k$, and apply $[\phi]_F$ to the constant cube $c = \square^k(x)$ in $C^k(X)$.  Then $[\phi]_F . c$ is a cube, agreeing with $c$ on all but the topmost vertex of $\{0,1\}^k$, and so is equal to $c$ by $k$-uniqueness.  Hence $\phi(x) = x$.
\end{proof}

Note that in general we may have that $\Aut_0(X) \ne \Aut_1(X)$, and hence even if $X$ is a nilspace of degree $s$, this filtration is not proper, hence it does not guarantee that $\Aut_0(X)$ is a nilpotent group: indeed, we will see below that it usually is not.  However, it \emph{does} guarantee that $\Aut_1(X)$ is nilpotent (of nilpotency class at most $s$) under these conditions.  Hence, the group of $1$-translations, equipped with the filtration
\begin{equation}\label{eq:trans-filtration}
  \Aut_1(X) = \Aut_1(X) \supseteq \Aut_2(X) \supseteq \dots \supseteq \Aut_k(X) \supseteq \dots
\end{equation}
is a much better candidate for our group $G$ used to construct the nilmanifold $G / \Gamma$.  We will never have much need to consider the full automorphism group $\Aut(X)$ again, restricting our attention to $\Aut_1(X)$.
In what follows, we denote by $\Aut_\bullet(X)$ the group $\Aut_1(X)$ endowed with the
filtration \eqref{eq:trans-filtration}.

\begin{example}
  Suppose $X = (\RR/\ZZ)^2$, with cubespace structure coming from the usual degree $1$ filtration.  Then $\Aut(X) \cong \GL_2(\ZZ) \ltimes (\RR/\ZZ)^2$, the group of \emph{affine automorphisms} of $(\RR/\ZZ)^2$, acting by
  \[
    (M,a) . x = M x + a \ .
  \]
  Indeed, since the $k$-cubes of $X$ are the $k$-dimensional parallelepipeds, i.e.~configurations of the form
  \begin{align*}
    \{0,1\}^k &\to (\RR/\ZZ)^2 \\
    \omega &\mapsto x + \sum_{i=1}^k \omega_i h_i
  \end{align*}
  for some coefficients $x, h_i \in (\RR/\ZZ)^2$, it is easy to see that such maps send cubes to cubes. Conversely, since the $2$-cubes of $X$ are the configurations $[[x, y],[z,w]]$ such that $x+w=y+z$, we have that for any $\phi \in \Aut(X)$ and $x,y,z,w \in X$ such that $x+w=y+z$,
  \[
    \phi(x) + \phi(w) = \phi(y) + \phi(z)
  \]
  and so in particular taking $z=0$ we have
  \[
    (\phi(x+w) - \phi(0)) = (\phi(x) - \phi(0)) + (\phi(w) - \phi(0))
  \]
  so we deduce that $(\phi - \phi(0))$ is necessarily linear and hence $\phi$ is affine-linear.

    The subgroup of $1$-translations is precisely $\Aut_1(X) = \{\id\} \ltimes (\RR/\ZZ)^2$.  To see this, we observe that if $\phi \in \Aut^1(X)$ and $[[x,y],[z,w]] \in C^2(X)$ then so is $[[x,y],[\phi(z), \phi(w)]]$; equivalently, if $x,y,z,w \in (\RR/\ZZ)^2$ satisfy $x-y=z+w$ then $x-y=\phi(z)-\phi(w)$.  So, $\phi(z)-\phi(w) = z - w$ for all $z,w \in (\RR/\ZZ)^2$ and hence $\phi$ is just a translation $x \mapsto x+t$.

    It is easy to check that such maps obey the $1$-translation property for higher-dimensional cubes.

    For $k \ge 2$, $\Aut_k(X)$ is the trivial group.
\end{example}

\begin{example}
  Let $X = G / \Gamma$ where $G$ is the Heisenberg group
  \[
    G = \left\{ \heis{x}{y}{z} \colon x, y, z \in \RR \right\}
  \]
  with its central series filtration, and $\Gamma$ is
  \[
    \Gamma = \left\{ \heis{x}{y}{z} \colon x, y, z \in \ZZ \right\}  \ ,
  \]
  the usual discrete co-compact subgroup.

  Then for any $g \in G$, the map
  \[
    x \Gamma \mapsto g x \Gamma
  \]
  is a $1$-translation on $G / \Gamma$.  If $g$ lies in $G_2$ (the center of $G$) then this map is moreover a $2$-translation. It turns out in this case that these are the only elements of $\Aut_1(X)$ and $\Aut_2(X)$.

However, if we replace $\Gamma$ by
  \[
    \Gamma' = \left\{ \heis{x}{y}{z} \colon x, y \in 2 \ZZ, \, z \in \ZZ \right\}
  \]
  then the map
  \[
    \heis{x}{y}{z} \Gamma' \mapsto \heis{x}{y+1}{z+x}\Gamma'=\heis{x}{y}{z}\heis{0}{1}{0}\Gamma'
  \]
  is a $1$-translation not of the form $x \Gamma' \mapsto g x \Gamma'$.  However, the \emph{connected component of the identity} in $\Aut_1(G/\Gamma')$ still consists entirely of maps $x \Gamma' \mapsto g x \Gamma'$.

  The full automorphism group $\Aut(X)$ is rather complicated (again, one can argue that it contains a copy of $\SL_2(\ZZ)$) and will not concern us.
\end{example}

\begin{remark}
  \label{rem:generic-nilmanifold-auts}
  Generalizing the previous example, for any filtered nilmanifold $G /\Gamma$ it turns out that maps $x \Gamma \mapsto g x \Gamma$ are always in $\Aut_k(X)$ provided $g \in G_k$.  As we saw above, these need not be all the elements of $\Aut_k(G/\Gamma)$, but they will yield all of the connected component of the identity in $\Aut_k(G/\Gamma)$.

  This is of course encouraging for the approach of using $\Aut_1(X)$ as a proxy for $G$.
\end{remark}

The definition of $k$-translation we have given is in some sense the most natural, but can be cumbersome in practice, because of the need to consider rather general configurations, and in particular cubes of arbitrarily large dimension.

The following proposition establishes a convenient alternative definition that is equivalent under certain circumstances.

\begin{proposition}
  \label{prop:k-translation-equiv}
  Suppose $X$ is a \uppar{compact, ergodic} nilspace of degree $s$, and fix $k$, $0 \le k \le s+1$.  Then $\phi \in \Homeo(X)$ is a $k$-translation if and only if for any $c \in C^{s+1-k}(X)$ the configuration $\llcorner^k(c; \phi(c))$ is an $(s+1)$-cube.
\end{proposition}

Note this condition is a special case of that from Definition \ref{def:k-translation}, specializing to $n=s+1$ and $c = \square^k(c')$ for some $c' \in C^{s+1-k}(X)$.  Hence, the ``only if'' direction is clear; the content is that it suffices to check the definition on configurations of this form.

Before we give the proof we recall some facts and terminology from \cite{GMV1}.
We say that a cubespace obeys the {\em glueing axiom} if the following holds:
for any triple $c_1,c_2,c_3$ of cubes, $[c_1,c_2]$ and $[c_2,c_3]$
being cubes imply that $[c_1,c_3]$ is a cube, also.
This can be visualized as the cubes $[c_1,c_2]$ and $[c_2,c_3]$ glued
along the common face $c_2$.
We recall from \cite{GMV1}*{Proposition 6.2} 
that nilspaces obey the glueing axiom.
Readers unfamiliar with these ideas are advised to consult
\cite{GMV1}*{Section 6.1}. 

\begin{proof}
  We first recall \cite{GMV1}*{Proposition 3.11} 
  that in a nilspace of degree $s$, a configuration $\{0,1\}^n \to X$ for $n \ge s+1$ is a cube, if and only if every face of dimension $(s+1)$ is a cube.

 Our first objective is to prove that it is enough to prove the condition from Definition \ref{def:k-translation} for cubes $c$ of dimension $s+1$.

  Let $F$  be a face of $\{0,1\}^n$ of codimension $k$ and $c \in C^n(X)$ for some $n \ge s+1$. Every face of $[\phi]_F(c)$ of dimension $(s+1)$ has the form $[\phi]_{F'}(c')$ where $c'$ is a face of $c$ and $F'$ is a face of $\{0,1\}^{s+1}$ of codimension at most $k$.  So, the condition from Definition \ref{def:k-translation} for $n=s+1$ implies all cases $n \ge s+1$.

  If $n < s+1$ and $c \in C^n(X)$, we note that the duplicated configuration $\square^{s+1-n}([\phi]_F(c))$ has the form
  $[\phi]_{F'}(\square^{s+1-n}(c))$ required by Definition \ref{def:k-translation}, where $F' = \square^{s+1-n}(F)$ (abusing notation somewhat) is a corresponding face of $\{0,1\}^{s+1}$ of codimension $k$, and so this configuration lies in $C^{s+1}(X)$.
  Restricting to an appropriate face, we recover that $[\phi]_F(c) \in C^n(X)$.  So, the condition for $n=s+1$ implies the condition for all $n$.

  We now have to show that it suffices to consider cubes of the form $\square^k(c')$ for $c' \in C^{s+1-k}(X)$.  This is by a ``glueing argument'', which is very much related to the ``universal replacement property'' of the canonical equivalence relation in \cite{GMV1}*{Proposition 6.3}. 
There we prove that in a fibrant cubespace $Y$ the fact that $\llcorner^k(y;y')$ is a cube implies the following: if $c\in C^k(Y)$ is such that $c(\vec 1)=y$, then the configuration $c'$ defined by $c'(\omega)=c(\omega)$ for $\omega\neq\vec1$ and $c'(\vec1)=y'$, is also a cube. A full proof of this is given in the second claim in the proof of Proposition 7.12 in \cite{GMV1}. 
One can turn that into a proof of our claim by replacing vertices by cubes of dimension $s+1-k$ in a straightforward manner.

  Here we only give some diagrams for the case $s=1$ and $k=2$ for the reader's convenience. The general case is very similar but notationally awkward.

  Let $c \in C^{s+1}(X) = C^2(X)$.  We draw this as
  \inlinetikz{
    \singlesquare{$c(00)$}{$c(10)$}{$c(01)$}{$c(11)$}
  }
  and we wish to show that
  \inlinetikz{
    \singlesquare{$c(00)$}{$c(10)$}{$c(01)$}{$\phi(c(11))$}
  }
  is a cube.  Our hypothesis states that for every $a\in C^{s+1-k}(X) = C^0(X)=X$ the following configuration is an $(s+1)$-cube, i.e.~a $2$-cube:
  \inlinetikz{
    \singlesquare{$a$}{$a$}{$a$}{$\phi(a)$}
  }

    We consider
  \inlinetikz
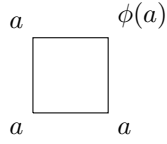
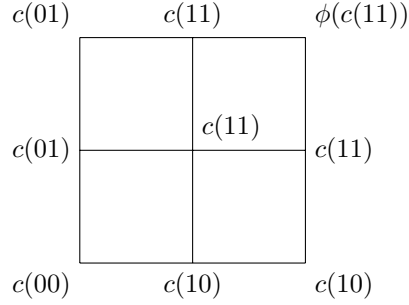
{
    \begin{scope}[scale=1.5]
      \draw[thin,black] (0,0) grid (2, 2);
      \node [below left] at (0, 0) {$c(00)$};
      \node [above left] at (0, 2) {$c(01)$};
      \node [below right] at (2, 0) {$c(10)$};
      \node [above right] at (2, 2) {$\phi(c(11))$};
      \node [left] at (0, 1) {$c(01)$};
      \node [right] at (2, 1) {$c(11)$};
      \node [below] at (1, 0) {$c(10)$};
      \node [above] at (1, 2) {$c(11)$};
      \node [above right] at (1, 1) {$c(11)$};
    \end{scope}
   }%
and note that the upper right square is a cube by this hypothesis, the remaining small squares are cubes by the cubespace axioms, and hence the outer square is a cube by glueing (twice) as required.
\end{proof}

Recalling Remark \ref{rem:generic-nilmanifold-auts}, the connected component of the identity in $\Aut_k(X)$ may be a better object to work with than $\Aut_k(X)$ itself.  We consider this object briefly now.

\begin{definition}
  For any topological group $G$, we let $G^\circ \le G$ denote the connected component of the identity in $G$.  In particular, $\Aut_k^\circ(X)$ denotes the connected component of the identity in the group of $k$-translations.
\end{definition}

\begin{remark}
  Note that the connected component is taken in $\Aut_k(X)$; \emph{not} in the larger group $\Aut(X)$ or $\Aut_1(X)$.  Indeed, these need not agree in general, i.e.~we may have $\Aut_k^\circ(X) \subsetneq \Aut_k(X) \cap \Aut_1^\circ(X)$.
\end{remark}

We should really check that this revised sequence of groups $\Aut_1^\circ(X) \supseteq \Aut_2^\circ(X) \supseteq \dots$ is still well-behaved.

\begin{proposition}
  Let $X$ be a compact cubespace.  Then the sequence $(\Aut_k^\circ(X))_{k \ge 0}$ is still a filtration of closed groups, in the sense of Proposition \ref{translation-prefiltration}.
\end{proposition}
\begin{proof}
  Since $\Aut_{k+1}^\circ(X)$ is a connected subset of $\Aut_k(X)$ containing the identity, it is contained in $\Aut_k^\circ(X)$.  Certainly all these groups are closed.  Since the commutator map is continuous, the set
  \[
    \{ [g,h] \colon g \in \Aut_i^\circ(X), g \in \Aut_j^\circ(X) \}
  \]
  is a connected subset of $\Aut_{i+j}(X)$, so is contained in $\Aut_{i+j}^\circ(X)$, and hence so is the closed subgroup it generates.
\end{proof}

\subsection{A discrete subgroup and the evaluation map}

Having defined the group $\Aut_1^\circ(X)$ that we hope will take the role of $G$ in the construction of the nilmanifold $G / \Gamma$, we now consider the question of how $\Gamma$, and the isomorphism $X \leftrightarrow G / \Gamma$, will arise.

Suppose we fix an element $x_0 \in X$.  There is a natural map
\begin{align*}
  \ev_{x_0} \colon \Aut_1(X) &\to X \\
  \phi &\mapsto \phi(x_0)
\end{align*}
and this gives rise to an identification of the orbit of $x$ under $\Aut_1(X)$ with $\Aut_1(X) / \stab(x_0)$ where $\stab(x) = \{ \phi \in \Aut_1(X) \colon \phi(x) = x \}$ denotes the stabilizer.  Clearly the same story makes sense restricting to $\Aut_1^\circ(X)$. For notational convenience we denote by $\stab(x)$ also $\{ \phi \in \Aut_1^\circ(X) \colon \phi(x) = x \}$. The meaning of $\stab(x)$ is always clear from the context.

Suppose we knew that $\stab(x_0)$ were a discrete and co-compact subgroup of $\Aut_1^\circ(X)$, and also that the action of $\Aut_1^\circ(X)$ on $X$ were transitive.  Then we would have a homeomorphism $X \cong \Aut_1^\circ(X) / \stab(x_0)$, which is already enough to identify $X$ -- as a topological space -- with a nilmanifold.

However, this transitivity assumption is a very significant one.  So far, we have made no progress towards even showing that $\Aut_1(X)$ is non-trivial.  Showing that $1$-translations (and more generally, $k$-translations) are fairly abundant, is both the core and the hardest aspect of the whole argument.

If the identification of $X$ with $\Aut_1^\circ(X) / \stab(x_0)$ should hold true in the category of cubespaces, we would need to know that $C^n(X)$ were identified under this bijection with the Host--Kra cubes $\HK^n(\Aut^\circ_\bullet(X)) / \stab(x_0)$ where here $\stab(x_0)=\{ \phi \in \HK^n(\Aut^\circ_\bullet(X)) \colon \phi(\square^n(x_0)) = \square^n(x_0) \}$.  Unwrapping the definitions, this says that a configuration $c \colon \{0,1\}^n \to X$ is a cube if and only if it has the form $\omega \mapsto \phi_\omega . x_0$ for some $(\phi_\omega) \in \HK^k(\Aut_\bullet^\circ(X))$.

It is certainly not hard to show the ``if'' direction (see also footnote 6).

\begin{proposition}
  \label{prop:evaluation-is-morphism}
  Suppose $(\phi_\omega)_{\omega \in \{0,1\}^n} \in \HK^n(\Aut_\bullet^\circ(X))$ is an element of the Host--Kra cube group, and $c \in C^n(X)$ is a cube.  Then the configuration
  \[
    \omega \mapsto (\phi_\omega(c(\omega)))
  \]
  is a cube in $C^n(X)$.
\end{proposition}
In particular, the map $\Aut_1^\circ(X) / \stab(x_0) \to X$ is a cubespace morphism.
\begin{proof}
  The Host--Kra cube group is generated (by definition) by elements $[\phi]_F$ where $F \subseteq \{0,1\}^n$ has codimension $k$ and $\phi \in \Aut_k^\circ(X)$.  By definition of $\Aut_k$, applying any such configuration pointwise to a cube $c \in C^n(X)$ yields another cube in $C^n(X)$.  The result follows from repeated application of this fact.
\end{proof}

However, the ``only if'' direction is not clear.  Indeed, the case $n = 0$ corresponds to the statement that the action of $\Aut_1^\circ(X)$ on $X$ is transitive; for higher $n$, this statement is some kind of assertion of ``higher transitivity'' or ``transitivity on all levels''.  Proving this is strongly analogous to proving transitivity, although one needs an analogous strengthening of the hypotheses.

\subsection{Structural results}
\label{statements-subsec}

We are now in a position to state more precise versions of the main structural result (Theorem \ref{main-thm-simple}) whose proof will occupy us for the rest of this article.

The first is a very slight strengthening of Theorem \ref{thm:main-toral}, using the different connectivity hypothesis explained above.  Part of our proof was reported in \cite{Gut_Oberwolfach}.

\begin{theorem}
  \label{thm:main-sc}
  Let $X = (X, C^n(X))$ be a \uppar{compact, ergodic} nilspace of degree $s$.  Suppose that all the structure groups $A_k(X)$ are Lie groups.  Finally, suppose $X$ is strongly connected.

  Fix an element $x_0 \in X$.  Then the following hold.
  \begin{enumerate}
    \item The group $\Aut_1^\circ(X)$ is Lie, and the filtration $(\Aut_k^\circ(X))_{k \ge 1}$ is Lie and
    has degree at most $s$.
    \item The subgroup $\stab(x_0) \le \Aut_1^\circ(X)$ is discrete and co-compact in $\Aut_1^\circ(X)$.  Moreover, $\stab(x_0) \cap \Aut_k^\circ(X)$ is co-compact in $\Aut_k^\circ(X)$ for all $k\geq 1$.
    \item The natural map
      \begin{align*}
        \Aut_1^\circ(X) &\to X \\
        \phi &\mapsto \phi(x_0)
      \end{align*}
      gives rise to an isomorphism of cubespaces $\Aut_1^\circ(X) / \stab(x_0) \to X$.  That is, this map is a homeomorphism, and $C^n(X)$ is identified with the Host--Kra cubes $\HK^n(\Aut_\bullet^\circ(X)) / \stab(x_0)$ on the nilmanifold $\Aut_1^\circ(X) / \stab(x_0)$.
  \end{enumerate}
  One can summarize these conclusions more briefly by saying that $X$ is a nilmanifold of degree at most $s$.
\end{theorem}

By Proposition \ref{prop-toral-sc}, this implies Theorem \ref{thm:main-toral}.

As mentioned above, it is instructive to record some of what our method provides when the connectivity hypothesis is not assumed.

\begin{theorem}
  \label{weaker-main-thm}
  Let $X = (X, C^n(X))$ be a compact ergodic nilspace of degree $s$, whose structure groups are Lie groups.  Then $X$ has finitely many connected components, each of which is open in $X$, and each of which is homeomorphic to a nilmanifold.

  Furthermore, for any $x_0 \in X$, the natural evaluation map
  \begin{align*}
    \ev_{x_0} \colon \Aut_1^\circ(X) &\to X \\
    \phi &\mapsto \phi(x_0)
  \end{align*}
  induces an homeomorphism $S \cong \Aut_1^\circ(X) / \stab(x_0)$, where $\Aut_1^\circ(X)$ is a nilpotent Lie group of degree at most $s$, $\stab(x_0)$ is a discrete and co-compact subgroup and $S$ denotes the connected component of $x_0$ in $X$.

  Finally, if $\Aut_1(X)$ is known \emph{a priori} to act transitively on the set of connected components of $X$, then there is a stronger identification \uppar{of topological spaces} $X \cong \Aut_1(X) / \stab(x_0)$ for any $x_0 \in X$, where again $\Aut_1(X)$ is a nilpotent Lie group of degree at most $s$ and $\stab(x_0)$ is a discrete co-compact subgroup.
\end{theorem}

Note this result provides information only about the topological structure of the nilspace $X$.
We are able to say something about the cubespace structure, as well, but we postpone this
discussion.
See Lemma \ref{lem:cube-eval-open} below.

The primary interest for the last statement of the above theorem comes from the case of dynamical systems, where this condition arises naturally.  In particular we record the following corollary.

\begin{corollary}\label{cor-weaker-main-thm}
  Let $(G,X)$ be a minimal topological dynamical system, where $X$ is a \uppar{compact, ergodic} nilspace of degree $s$, whose structure groups $A_k(X)$ are Lie groups; and suppose $G$ acts on $X$ through a continuous group homomorphism $\sigma \colon G \to \Aut_1(X)$.

  Then $X$ is homeomorphic to a nilmanifold as in the previous theorem: that is, $X \cong \Aut_1(X) / \stab(x_0)$ for any $x_0 \in X$.  Under this identification, $G$ acts by left translations $x \Gamma \mapsto \sigma(g) x \Gamma$.
\end{corollary}
\begin{proof}
  Since $(G,X)$ is minimal, one may map any element of $X$ into any open set by an element of $G$.  But the connected components of $X$ are open by the theorem, so we conclude that the group of translations $\Aut_1(X)$ acts transitively on the set of connected components of $X$.
\end{proof}

To motivate the above result further, we mention that for any minimal group action
$(G,X)$, one can associate a natural cubespace structure $\{C_G^{n}\}_{n\in\ZZ_{\ge 0}}$
such that the acting group $G$ immerses into $\Aut_1(X,C_G)$.
This construction is due to Host, Kra and Maass \cite{HKM10}.
Given a minimal group action $(G,X)$, Gutman, Glasner and Ye \cite{GGY} defined the regional proximal
relation $\RP_G^n(X)$, which is a closed equivalence relation on $X$ for each $n\in\ZZ_{\ge 0}$, and extends the definition of
Host, Kra and Maass \cite{HKM10} for actions of $\ZZ$ to general group actions.
It is proved in \cite{GGY} that the cubespace structure $(X,C_G)$ is a nilspace
of degree at most $s$  for some $s\in\ZZ_{\ge 0}$ if and only if $\RP_G^s(X)$ is trivial. 
For a detailed exposition and further results we refer to the companion paper \cite{GMV3} and to \cite{GGY}.

Based on these results, the following is a special case of Corollary \ref{cor-weaker-main-thm}.

\begin{corollary}\label{cor:even-more-weaker}
  Let $(G,X)$ be a minimal topological dynamical system.
  Suppose that the regional proximal relation $\RP_G^s(X)$ is trivial for some $s\in\ZZ_{\ge 0}$
  and the structure groups of $(X,C_G)$ are Lie groups.
 
  Then $(G,X)$ is a nilsystem, that is, $X$ can be identified with a nilmanifold $N/\Gamma$, where $N$
  is a nilpotent Lie group of step at most $s$ and the action of $G$ can be realized through a continuous homomorphism
  $G\to N$.
\end{corollary}

We note that the hypothesis on the structure groups can be replaced by some assumptions on the topology of $X$
similarly to  Theorem \ref{main-thm-simple}, and therefore, the corollary can be stated without any reference to the cubespace structure.
However, the proof of that result requires the structure theory for general nilspaces developed in \cite{GMV3}, and we refer
the interested reader to the Appendix of that paper.

\section{The main steps in the structural result}
\label{sec:main-outline}

We continue our outline of the proof of the main structure theorems of this paper
following \cite{CS12}.
From the discussion above, it is clear that our main task is to show that $X$ has ``enough $k$-translations'' in some sense.  At this stage, it is not particularly obvious that even a single non-trivial $k$-translation exists for any $k$, and so we first need to find a source of these automorphisms.

The source is ultimately based on the weak structure theory expounded in \cite{GMV1}.  We recall that if $X$ is a (compact, ergodic) nilspace of degree $s$, then there is a \emph{canonical factor}
\[
  \pi_{s-1} \colon X \to \pi_{s-1}(X)
\]
where $\pi_{s-1}(X)$ is a (compact, ergodic) nilspace of degree at most $(s-1)$, and the fibers of $\pi_{s-1}$ are identified with the $s$-th structure group $A_s(X)$ of $X$, which is some compact abelian group.  More precisely, there is an action of $A_s(X)$ on the whole space $X$, whose orbits are precisely these fibers.  For full statements, see \cite{GMV1}*{Theorem 5.4}. 

The key observation is the following.
\begin{proposition}
  \label{prop:structure-group-translations}
  Let $X$ be a compact ergodic nilspace of degree $s$, and let $a \in A_s(X)$.  Write $t_a \colon X \to X$ for the action $x \mapsto a . x$ of $A_s$ on $X$ described above.

  Then $t_a \in \Aut_s(X)$.
\end{proposition}
\begin{proof}
  After unwrapping the definitions, this is immediate from the weak structure theory.  Indeed, what we need -- that applying a fixed element $a \in A_s$ to a face of $c \in C^n(X)$ of codimension $s$ yields another cube -- is a special case of
  \cite{GMV1}*{Theorem 5.4(ii)}. 
\end{proof}

So, although we still cannot prove that $\Aut_1(X)$ acts transitively, we at least know that it acts transitively on each fiber of $\pi_{s-1}$ (and in fact this is true even of $\Aut_s(X)$).

The second idea is then to induct on the degree.  Suppose we had already established the relevant transitivity statement for nilspaces of degree $(s-1)$; so in particular we knew that $\Aut_1^\circ(\pi_{s-1}(X))$ acts transitively on $\pi_{s-1}(X)$ (given suitable assumptions).  Then we might hope to achieve transitivity on $X$ by first moving to the correct fiber of $\pi_{s-1}$ (by inductive assumption), and then moving to the correct point in that fiber using an element of $A_s$.

Of course, the challenge for this plan is that there is no obvious way to relate $\Aut_1(\pi_{s-1}(X))$ to $\Aut_1(X)$.  Specifically, to make this approach work we will need to be able to \emph{lift} a $1$-translation of $\pi_{s-1}(X)$ to one of $X$, under suitable hypotheses; i.e.~given $\phi \in \Aut_1(\pi_{s-1}(X))$, to find $\tilde{\phi} \in \Aut_1(X)$ such that the diagram
\[
  \begin{CD}
    X @>\tilde{\phi}>> X \\
    @VV\pi_{s-1}V @VV\pi_{s-1}V \\
    \pi_{s-1}(X) @>\phi>> \pi_{s-1}(X)
  \end{CD}
\]
commutes.

Going in the other direction -- i.e.~that given $\tilde{\phi} \in \Aut_k(X)$ we can push it down to $\phi \in \Aut_k(\pi_{s-1}(X))$ -- is straightforward, and we verify this now.

\begin{proposition}
  Let $X$ be a compact ergodic nilspace of degree $s$, and fix $k \ge 0$.  Then there is a canonical \uppar{continuous} group homomorphism $\pi_\ast \colon \Aut_k(X) \to \Aut_k(\pi_{s-1}(X))$ such that
  \[
    \begin{CD}
      X @>\phi>> X \\
      @VV\pi_{s-1}V @VV\pi_{s-1}V \\
      \pi_{s-1}(X) @>\pi_\ast(\phi)>> \pi_{s-1}(X)
    \end{CD}
  \]
  commutes for every $\phi \in \Aut_k(X)$.
\end{proposition}
The proof is completely mechanical.
\begin{proof}
  Note that if $k > s$ the result is trivial (Proposition \ref{prop-translations-terminate}), so we assume $k \le s$.  We do the only thing possible: given $y \in \pi_{s-1}(X)$, we pick $x \in \pi_{s-1}^{-1}(y)$ arbitrarily, and define $(\pi_\ast(\phi))(y) := \pi_{s-1}(\phi(x))$.  It suffices to check that this is well-defined and a $k$-translation; continuity is then straightforward to check.

  If $x' \in \pi_{s-1}^{-1}(y)$ is another lift of $y$ then there is some $a \in A_s$ such that $x' = a . x$ (see \cite{GMV1}*{Theorem 5.4(i)}). 
  But $t_a$ is an element of $\Aut_s(X)$ (Proposition \ref{prop:structure-group-translations}), and this commutes with $\phi$ (Proposition \ref{prop-translations-terminate} again, and Proposition \ref{translation-prefiltration}).  So, $\phi(x') = a . \phi(x)$, and so $\pi_{s-1}(\phi(x')) = \pi_{s-1}(a . \phi(x)) = \pi_{s-1}(\phi(x))$, since $t_a$ preserves fibers of $\pi_{s-1}$.  This shows well-definedness.

  Now suppose $c \in C^n(\pi_{s-1}(X))$ and $F \subseteq \{0,1\}^n$ a face of codimension $k$ are given.  By definition, there is a cube $\tilde{c} \in C^n(X)$ such that $\pi_{s-1}(\tilde{c}) = c$.  So,
  \[
    [\pi_\ast(\phi)]_F . c = \pi_{s-1}([\phi]_F . \tilde{c})
  \]
  and the right hand side is a cube as required, since $\phi$ is a $k$-translation and $\pi_{s-1}$ is a cubespace morphism.
\end{proof}

The statement for identity components follows trivially from this, i.e.~$\pi_\ast(\Aut_k^\circ(X)) \subseteq \Aut_k^\circ(\pi_{s-1}(X))$.

The ``translation lifting'' statement we require is precisely that this group homomorphism is surjective.  This is not quite true in general.  However, it is the case that any \emph{sufficiently small} element of $\Aut_k(\pi_{s-1}(X))$ can be lifted to a $k$-translation on $X$, i.e.~lies in the image of $\pi_\ast$.  In fact we will show the following.

\begin{proposition}
  \label{lem:surj}
  Let $X$ be a \uppar{compact, ergodic} nilspace of degree $s$ whose structure groups $A_r(X)$ are all Lie. For any $k \ge 1$, the homomorphism $\pi_\ast \colon \Aut_k(X) \to \Aut_k(\pi_{s-1}(X))$ is an open map.
\end{proposition}

In particular, the image of $\pi_\ast$ is an open subgroup of $\Aut_k(\pi_{s-1}(X))$, meaning exactly that all sufficiently small translations on $\pi_{s-1}(X)$ lift to $X$.  More precisely: there exists $\delta > 0$ (depending only on $X$ and $k$) such that if $\phi \in \Aut_k(\pi_{s-1}(X))$ is a ``perturbation of the identity'' in the sense that $\|\phi\| \le \delta$, then there is a lift $\tilde{\phi} \in \Aut_k(X)$ such that $\pi_\ast(\tilde{\phi}) = \phi$, i.e.~$\pi_{s-1} \circ \tilde{\phi} = \phi \circ \pi_{s-1}$.

Proposition \ref{lem:surj} is the heart of the argument, and will occupy us for the majority of the rest of the paper.  We will return to its proof in Section \ref{sec-main-details}.

\bigskip

For the remainder of this section, we will fill in the outstanding gaps in the deduction of Theorem \ref{thm:main-sc}, and the other results stated in Section \ref{statements-subsec}, from Proposition \ref{lem:surj}.

Some of this work is concerned with making the previous discussion rigorous, and also extending all mention of transitivity of the action of $\Aut_1(X)$ to the ``higher transitivity'' required to get the corresponding statement about cubes.

A large part of the rest is of a strongly topological nature: we will have to check that certain groups are Lie groups, certain maps are open maps and so forth.  To some extent, this material is technical and could be skipped on first reading.  However, we caution that it cannot easily be separated from the rest of the argument: without this topological input, we could not prove even a much weakened version of Theorem \ref{thm:main-sc}.

Here and later in the paper we will have to draw on a couple of fairly powerful results about topological groups: Gleason's lemma (Theorem \ref{gleason-lemma}) concerning principal bundles, and the Gleason--Kuranishi extension theorem (Theorem \ref{thm:gleason-extension}) that asserts that an extension of a Lie group by a Lie group is Lie.

However, we remark that we will not use the full power of either of these.  Specifically, Gleason's lemma is only required in the compact abelian Lie setting, where the proof is essentially an application of Fourier analysis.  For the Gleason--Kuranishi extension theorem, we require only the case of central extensions.  Both of these facts are given a self-contained treatment in \cite{T14} which is substantially easier than the original references. See Remark \ref{remark-tao} for more details.

\subsection{The stabilizer is discrete}
\label{subsec-stab-discrete}

To construct our nilmanifold $G/\Gamma$, recall we plan to use $G = \Aut_1^\circ(X)$ and $\Gamma = \stab(x_0)$ for some $x_0 \in X$.  So, we will need to know that this $\Gamma$ is discrete.

\begin{lemma}
  \label{stab-discrete}
  Let $X$ be a \uppar{compact, ergodic} nilspace of degree $s$ whose structure groups $A_r(X)$ are Lie, and fix any $x_0 \in X$.  Then $\stab(x_0) \subseteq \Aut_1(X)$ is discrete.
\end{lemma}

Equivalently (since $\Aut_1(X)$ carries the supremum metric) we wish to show that if $\phi$ is a $1$-translation other than the identity and $\phi(x_0) = x_0$ then $\|\phi\| = \sup_{x \in X} d(x, \phi(x)) \ge \delta$ for some constant $\delta$ independent of $\phi$ (and, in fact, $x_0$).

\begin{proof}[Proof of Lemma \ref{stab-discrete}]
  We proceed by induction on $s$.  The case $s=0$ is a triviality (since a $0$-step compact ergodic nilspace is just the one-point space $\{\ast\}$).

  Again we consider $\pi_\ast \colon \Aut_1(X) \to \Aut_1(\pi_{s-1}(X))$.  It is clear from the definition that $\pi_\ast$ maps $\stab(x_0)$ into $\stab(\pi_{s-1}(x_0))$.  By inductive hypotheses, $\stab(\pi_{s-1}(x_0))$ is discrete; so we conclude that if $\phi \in \Aut_1(X)$ and $\|\phi\|$ is sufficiently small, then $\phi \in \ker(\pi_\ast)$.  Equivalently, $y$ and $\phi(y)$ lie in the same fiber for all $y \in X$.

  By the weak structure theory, for each $y \in X$ there is an unique $a \in A_s$ such that $\phi(y) = a . y$.  Hence there is a function $\tau \colon X \to A_s$ such that $\phi(y) = \tau(y) . y$.  It is clear $\tau$ is continuous.  Since $\phi$ commutes with the action of $A_s \subseteq \Aut_s(X)$ (by Proposition \ref{prop:structure-group-translations} and Proposition \ref{translation-prefiltration}), in fact $\tau$ factors through $\pi_{s-1}$, i.e.~as a map $X \to \pi_{s-1}(X) \to A_s$.  By abuse of notation we write $\tau$ for this map $\pi_{s-1}(X) \to A_s$.

  We now want to unwrap the fact that $\phi$ is a $1$-translation to a condition on $\tau$.  By Proposition \ref{prop:k-translation-equiv}, this is equivalent to saying that $[c, \tau(c) . c] \in C^{s+1}(X)$ for every $c = C^s(X)$ (where $\tau(c)$ denotes pointwise application).  By the weak structure theorem \cite{GMV1}*{Theorem 5.4}, 
  this is equivalent to saying that
  \[
    [\square^s(0),\tau(c)] \in C^{s+1}(\cD_s(A_s(X)))
  \]
  for every $c \in C^n(X)$.

  We recall that $\cD_s(A_s)$ is the cubespace whose base space is the group $A_s$, and whose cubes are $\HK^n(A_s)$ where $A_s$ is given the degree $s$ filtration $A_s = \dots = A_s \supseteq \{0\}$.  We also recall (\cite{GMV1}*{Proposition 5.1}) 
  that a configuration $a \colon \{0,1\}^{s+1} \to A_s$ lies in $C^{s+1}(\cD_s(A_s))$ if and only if
  \[
    \sum_{\omega \in \{0,1\}^{s+1}} (-1)^{|\omega|} a(\omega) = 0 \ .
  \]
  It follows that $\phi$ is a $1$-translation if and only if the function
  \begin{align*}
    \rho \colon C^s(\pi_{s-1}(X)) &\to A_s \\
    c &\mapsto \sum_{\omega \in \{0,1\}^s} (-1)^{|\omega|} \tau(c(\omega))
  \end{align*}
  is identically zero.

  We will see in Section \ref{sec:cocycle} that this is a natural statement in terms of the rudimentary cohomology theory expounded there.  Moreover, if $\|\phi\|$ is sufficiently small (i.e.~less than some absolute constant), then $\tau$ also takes values close to $0$ in $A_s$ (i.e.~$d(0, \tau(y)) \le \eps$ for any fixed small $\eps$ and all $y$).  Under these conditions, a rigidity result, Theorem \ref{thm:cocycle-uniqueness}, states that $\tau$ must in fact be constant.

  Since $\phi(x_0) = x_0$ and hence $\tau(x_0) = 0$ by assumption, we must therefore have that $\tau$ is identically zero.  But then $\phi = \id$, as required.
\end{proof}

\subsection{The group \texorpdfstring{$\Aut_1(X)$}{Aut1(X)} is Lie}

Recall that we wish to use $\Aut_1^{\circ}(X)$ as the group $G$ in the definition of a nilmanifold $G / \Gamma$. 
Hence it will be important to verify that this is a Lie group.
In fact, we will prove that $\Aut_1(X)$ is a Lie group, which implies that $\Aut_1^{\circ}(X)$ is also a Lie group.

Once again we stress that our definitions do not require a Lie group to be connected.

\begin{lemma}
  \label{lemma-auts-lie}
  Let $X = (X, C^n(X))$ be a \uppar{compact, ergodic} nilspace of degree $s$, whose structure groups $A_k(X)$ are Lie.  Then $\Aut_1(X)$ is a Lie group.
\end{lemma}

Our approach will be to apply induction on $s$, with the case $s=0$ a trivial base case.  We again consider the map
\[
  \pi_{\ast} \colon \Aut_1(X) \to \Aut_1(\pi_{s-1}(X)) \ .
\]
Since $\pi_{s-1}(X)$ is a compact ergodic nilspace of degree at most $(s-1)$, by inductive hypothesis $\Aut_1(\pi_{s-1}(X))$ is Lie. 
By Proposition \ref{lem:surj}, the image of $\pi_\ast$ is an open subgroup, and so is also Lie.

We will use the following result.
\begin{theorem}[{Gleason--Kuranishi extension theorem}]
\footnote{This is Theorem 3.1 of \cite{G51}. The Kuranishi extension theorem
usually refers to the same statement where it is assumed in addition
that $G$ is locally compact (see \cite{I_MR}). As pointed out by
Gleason, Kuranishi actually proved a weaker statement in \cite{K50}.}
  \label{thm:gleason-extension}
  Let $G$ be a topological group, and suppose there exists a closed normal subgroup $N$ of $G$ such that $N$ and $G/N$ are Lie groups.  Then $G$ is also Lie.
\end{theorem}

Given this, it suffices to verify that $\ker(\pi_\ast) \le \Aut_1(X)$ is a Lie group.  We know that the copy of $A_s(X)$ in $\Aut_1(X)$ (acting by $t_a \colon x \mapsto a . x$) is contained in $\ker(\pi_\ast)$, since it acts on each fiber of $\pi_{s-1}$.  Hence it will suffice to show:

\begin{lemma}
  The subgroup $A_s(X) \le \ker(\pi_\ast)$ is open.
\end{lemma}
\begin{proof}
  Suppose $\phi \in \ker(\pi_\ast) \setminus A_s$.  Fix any point $y \in X$.  Since $y$ and $\phi(y)$ lie in the same fiber of $\pi_{s-1}$, there is an unique $a \in A_s(X)$ such that $y = a . \phi(y)$.  As before write $t_a$ for the translation $x \mapsto a . x$ on $X$, and define $\phi'(y) = t_a \circ \phi(y)$; hence, $\phi'(x) = x$ and so $\phi' \in \stab(x) \subseteq \Aut_1(X)$.

  By Lemma \ref{stab-discrete}, $\|\phi'\| \ge \delta$ for some constant $\delta$ independent of $\phi$.  If $\|t_a\| \le \delta / 2$, then $\|\phi\| \ge \delta - \delta / 2 = \delta / 2$ and we are done.  If not, i.e.~$\|t_a\| > \delta/2$, then since $A_s(X)$ is compact and acts freely on $X$, we have that $d(x, t_a(x))$ is bounded \emph{below} by a constant depending only on $\delta/2$, and hence $\|\phi\|\ge d(x,t_a(x))$ is bounded below as required.
\end{proof}

We noted above (Remark \ref{rem:second-countable-auts}) that $\Aut(X)$ is second countable, and hence clearly so is the closed subgroup $\ker(\pi_\ast)$.
Moreover, $A_s(X)$ is an open subgroup of $\ker(\pi_\ast)$ that is Lie, and so by standard results in Lie theory we may extend the differentiable structure on $A_s(X)$ to one on all of $\ker(\pi_\ast)$, making the latter into a Lie group.
This completes the proof of Lemma \ref{lemma-auts-lie}.

\begin{remark}\label{remark-tao}
  Noting that the copy of $A_s(X)$ in $\Aut_1(X)$ is central, we have only really used the fact that a central extension of a Lie group by a Lie group is Lie, and that a \emph{discrete} extension of a Lie group is Lie.  The latter fact is straightforward given elementary Lie theory (e.g., assuming something of equivalent strength to Cartan's closed subgroup theorem).  Hence we may rely on the proof in  \cite{T14}*{Theorem 2.6.1} rather than \cite{G51}.

  In fact we know rather more, namely that the middle group $\Aut_1(X)$ in the exact sequence $0\to \ker(\pi_\ast) \to \Aut_1(X) \to \Aut_1(\pi_{s-1}(X))\to 0$ is nilpotent.  It is likely one can obtain a yet more direct and elementary proof of Theorem \ref{thm:gleason-extension} in this special case, but we have not pursued this.
\end{remark}

\subsection{Completing the proof of Theorem \ref{thm:main-sc}}

We now have all the ingredients in place to deduce Theorem \ref{thm:main-sc}, as well as the other results stated in Section \ref{statements-subsec}.

Recall that we wished to know that $\Aut_1(X)$ acts transitively on $X$, under suitable hypotheses.  Again it turns out the most natural statement of this is in terms of openness of some map.

\begin{lemma}
  \label{lem:open-evaluation-map}
  Let $X$ be a \uppar{compact, ergodic} nilspace of degree $s$ whose structure groups $A_r(X)$ are Lie, and let $x_0 \in X$ be fixed.  Then the evaluation map
  \begin{align*}
    \ev_{x_0} \colon \Aut_1(X) &\to X \\
    \phi &\mapsto  \phi(x_0)
  \end{align*}
  is continuous and open.
\end{lemma}
Equivalently, this states that if two points $x, y \in X$ are very close together, then there is a small $1$-translation $\phi$ such that $\phi(x) = y$.

For inductive reasons we will verify the following slight strengthening.

\begin{lemma}
  \label{lem:inductive-open-evaluation-map}
  Let $X$ be as in Lemma \ref{lem:open-evaluation-map}.  Let $0 \le k < s$ be fixed.  For all $\eps > 0$ there exists $\delta > 0$ such that the following holds: if $x, y \in X$, $d(x,y) \le \delta$ and $\pi_k(x) = \pi_k(y)$, then there exists $\phi \in \Aut_{k+1}(X)$, $\|\phi\| \le \eps$ such that $\phi(x) = y$.
\end{lemma}
\begin{proof}
  Essentially this just combines Proposition \ref{prop:structure-group-translations} with Proposition \ref{lem:surj}.  We proceed by induction on $s$, and note that the case $s=0$ is trivial.

  If $k = s-1$, then the result is immediate from Proposition \ref{prop:structure-group-translations}: we know $\pi_{s-1}(x) = \pi_{s-1}(y)$, so we can set $\phi = t_a$ for the unique $a \in A_s$ such that $a.x = y$, and so $\phi \in \Aut_s(X)$ as required.  The bound on $\|t_a\|$ follows from Proposition \ref{prop:robust-free-action}.

  If $k < s-1$, we apply induction on $s$ to obtain $\tilde{\phi} \in \Aut_{k+1}(\pi_{s-1}(X))$ such that $\tilde{\phi}(\pi_{s-1}(x)) = \pi_{s-1}(y)$; and $\|\tilde{\phi}\|$ can be made arbitrarily small if $d(x,y)$ is small enough.

  By Proposition \ref{lem:surj}, we may obtain a lift $\phi' \in \Aut_{k+1}(X)$ of $\tilde{\phi}$ (i.e.~such that $\pi_\ast(\phi') = \tilde{\phi}$) where again $\|\phi'\|$ can made as small as we like.

  We now know that $\pi_{s-1}(\phi'(x)) = \pi_{s-1}(y)$.  Finally, we note that $d(\phi'(x), y) \le d(x,y) + \|\phi'\|$ which is still as small as we like, so we again apply the case $k=s-1$ to obtain $\psi \in \Aut_{s}(X)$ such that $\psi(\phi'(x)) = y$.  So, $\phi := \psi \circ \phi'$ is in $\Aut_{k+1}(X)$ and is still arbitrarily small.
\end{proof}

\begin{proof}[{Proof of Lemma \ref{lem:open-evaluation-map}}]
  Continuity is immediate by the choice of metric on $\Aut_1(X)$.  Since $\pi_0(X) = \{\ast\}$, openness follows immediately from the case $k=0$ of Lemma \ref{lem:inductive-open-evaluation-map}.
\end{proof}

Our final conclusions are phrased in terms of the identity component $\Aut_1^\circ(X)$.  It is fairly painless to deduce useful facts about this from what we have already shown about $\Aut_1(X)$.

\begin{corollary}
  Let $X$ be as in the statement of Lemma \ref{lem:open-evaluation-map}.  Then the identity component $\Aut_1^\circ(X)$ acts transitively on each connected component of $X$.

  Moreover, there are only finitely many connected components of $X$, and each one is open and closed in $X$.
\end{corollary}
\begin{proof}
  Since $\Aut_1(X)$ is Lie (Lemma \ref{lemma-auts-lie}) its identity component $\Aut_1^\circ(X)$ is open.  Therefore by Lemma \ref{lem:open-evaluation-map}, the orbits of $\Aut_1^\circ(X)$ in $X$ are open, and hence closed (since the orbits partition the space).  It follows that any connected subset of $X$ is contained in a single orbit.

  Since each orbit of $\Aut_1^\circ(X)$ is the image of a connected set under a continuous map and hence connected, we further deduce that every point of $X$ has a connected open neighbourhood, and so all connected components are open.  Since they partition the space, they are therefore closed; and since $X$ is compact, it follows there are only finitely many.
\end{proof}

Combined with Lemma \ref{lemma-auts-lie} and Lemma \ref{stab-discrete}, this is enough to complete the proof of Theorem \ref{weaker-main-thm}.

\bigskip

We now turn to the outstanding statements from Theorem \ref{thm:main-sc}.  Essentially, we still have said nothing about the cubes $C^n(X)$ in terms of the filtration $\Aut_k(X)$.  To do so, we need to generalize several of the arguments above that concerned $1$-translations acting on $X$, to statements about the Host--Kra cube group $\HK^n(\Aut_\bullet(X))$ acting on $C^n(X)$.

For instance, the following generalizes Lemma \ref{lem:open-evaluation-map}.

\begin{lemma}
  \label{lem:cube-eval-open}
  Let $X$ be as in Lemma \ref{lem:open-evaluation-map}, and fix $c \in C^n(X)$.  Then the evaluation map
  \begin{align*}
    \ev_c \colon \HK^n(\Aut_\bullet(X)) &\to C^n(X) \\
    (\phi_\omega)_{\omega \in \{0,1\}^n} &\mapsto \left(\omega \mapsto \phi_\omega(c(\omega)) \right)
  \end{align*}
  is open.
\end{lemma}
\begin{remark}
  Recall that this definition makes sense by virtue of Proposition \ref{prop:evaluation-is-morphism}.
\end{remark}

Note this lemma (as well as the previous ones) requires no connectivity assumptions.
Hence this lemma provides a lot of information about cubes in not necessarily strongly connected nilspaces.
In particular it implies that all sufficiently small cubes can be obtained from a constant cube using the action
of the Host Kra cubegroups of $\Aut_\bullet(X)$.
On the other hand, we have no information about large cubes, this is why we need the connectivity
hypotheses in the main results.

The missing ingredients that prevent us from simply repeating the proof of Lemma \ref{lem:open-evaluation-map} are all concerned with topological facts about Host--Kra cube groups.  Everything we will need follows cleanly in turn from the algebraic theory of $\HK^n(G_\bullet)$ expounded in the appendix of \cite{GMV1}.
We recall these results now.

\begin{lemma}[\cite{GMV1}*{Lemma A.12}]\label{lm:hk-homomorphisms} 
Let $G_\bullet$, $H_\bullet$ be two filtered topological groups.
Let $\tau:G\to H$ be a homomorphism such that
$\tau(G_i)\subseteq\tau(H_i)$.
Then $\tau$ induces a homomorphism $\tau: \HK^n(G_\bullet)\to\HK^n(H_\bullet)$
for each $n$ by pointwise application on the vertices.

If $\tau:G_i\to H_i$ is open for each $i$, then so is the induced homomorphism
 $\tau: \HK^n(G_\bullet)\to\HK^n(H_\bullet)$
for each $n$.
\end{lemma}

\begin{lemma}[\cite{GMV1}*{Lemma A.13}]\label{lm:hk-subgroups} 
Let $G_\bullet$, $H_\bullet$ be two filtered topological groups.
Suppose $G_i\subseteq H_i$ for each $i$.
If $G_i$ are open in $H_i$ (resp., connected) for each $i$, then
$\HK^n(G_\bullet)$ is also open in $\HK^n(H_\bullet)$ (resp., connected) for each $n$.
\end{lemma}

We note the following consequence.
\begin{lemma}
  \label{cube-surj-lemma}
  Let $X$ be a \uppar{compact, ergodic} nilspace of degree $s$ whose structure groups $A_r(X)$ are Lie groups.  Let $\pi_{s-1} \colon X \to \pi_{s-1}(X)$ denote the canonical projection, and $\pi_\ast \colon \Aut_k(X) \to \Aut_k(\pi_{s-1}(X))$ the canonical homomorphism described previously.

  Then
  \[
    \pi_\ast \colon \HK^n(\Aut_\bullet(X)) \to \HK^n(\Aut_\bullet(\pi_{s-1}(X)))
  \]
  is an open map.
\end{lemma}
\begin{proof}
  This follows from Proposition \ref{lem:surj} and Lemma \ref{lm:hk-homomorphisms}.
\end{proof}

\begin{proof}[Proof of Lemma \ref{lem:cube-eval-open}]
  We follow the proof of Lemma \ref{lem:open-evaluation-map} very closely; in fact, that result corresponds precisely to the case $n=0$ of this.  As ever, we induct on $s$, the degree of $X$, the case $s=0$ being trivial.

  It suffices to show that for any $\eps > 0$ there exists a $\delta > 0$ such that if $c_1, c_2 \in C^n(X)$ and $d(c_1,c_2) \le \delta$ then there exists $\phi \in \HK^n(\Aut_\bullet(X))$, $\|\phi\| \le \eps$ such that $\phi(c_1) = c_2$.

  By inductive hypothesis, the map $\ev_{\pi_{s-1}(c)} \colon \HK^n(\Aut_\bullet(\pi_{s-1}(X))) \to C^n(\pi_{s-1}(X))$ is open, and by Lemma \ref{cube-surj-lemma} so is $\pi_\ast \colon \HK^n(\Aut_\bullet(X)) \to \HK^n(\Aut_\bullet(\pi_{s-1}(X)))$.  As before, we consider the composite and deduce that, for suitably chosen $\delta$, there exists $\phi' \in \HK^n(\Aut_\bullet(X))$, $\|\phi'\| \le \eta$ such that $\pi_{s-1}(\phi'(c_1)) = \pi_{s-1}(c_2)$, where $0 < \eta \le \eps / 2$ is some parameter to be determined.

  Now, we have that $d(\phi'(c_1), c_2) \le \delta + \eta$ and moreover these lie in the same fiber of $\pi_{s-1}$.  But we recall that the fibers of $C^n(X) \to C^n(\pi_{s-1}(X))$ are completely described by the weak structure theory (\cite{GMV1}*{Theorem 5.4(ii)}): 
   specifically, $C^n(\cD_s(A_s(X)))$ acts simply transitively on each fiber.  We again use $t_a$ to denote this action, and crucially observe that it corresponds to an element of $\HK^n(\Aut_\bullet(X))$ (since $C^n(\cD_s(A_s(X)))$ is defined as the Host--Kra group on $A_s(X)$ with the degree $s$ filtration, which is a sub-filtration of $\Aut_\bullet(X)$).

   So, there is an unique $a \in C^n(\cD_s(A_s(X)))$ such that $t_a(\phi'(c_1)) = c_2$.  By Proposition \ref{prop:robust-free-action}, for suitably chosen $\eta$ and $\delta$ we may ensure that $\|t_a\| \le \eps / 2$.  Hence taking $\phi = t_a \circ \phi'$ we have $\phi(x) = y$ and $\|\phi\| \le \eps$ as required.
\end{proof}

Finally we can conclude the proof of Theorem \ref{thm:main-sc}.
\begin{proof}[Proof of what is left of Theorem \ref{thm:main-sc}]
What is left to prove is that $\HK^n(\Aut_\bullet^\circ(X))$ acts transitively on $C^n(X)$, using in particular that $X$ is strongly connected.

Since $\Aut_k(X)$ is a closed subgroup of $\Aut_1(X)$, it is Lie (by Cartan's closed subgroup theorem) and hence $\Aut_k^\circ(X)$ is open in $\Aut_k(X)$.  By Lemma \ref{lm:hk-subgroups}, it follows that $\HK^n(\Aut_\bullet^\circ(X))$ is an open subgroup of $\HK^n(\Aut_\bullet(X))$.  Hence, by Lemma \ref{lem:cube-eval-open}, the orbits of the action of this group on $C^n(X)$ are open, and hence closed (as they partition the space).  But since $C^n(X)$ was assumed to be connected, this means the action is transitive, as required.

  Given this, it follows that the map $\ev_{\square^n(x_0)} \colon \HK^n(\Aut_\bullet^\circ(X)) \to C^n(X)$ given by $(\phi_\omega)_{\omega \in \{0,1\}^n} \mapsto \left(\omega \mapsto \phi_\omega(x_0) \right)$ is continuous, open and surjective, and so induces an homeomorphism $\HK^n(\Aut_\bullet^\circ(X))/\stab(x_0)\cong C^n(X)$ for each $n\geq 0$. In particular when $n=0$ we have that $\Aut_1^\circ(X)/\stab(x_0)\cong X$.

  Finally, we must check that $\stab(x_0) \cap \Aut_k^\circ(X)$ is co-compact in $\Aut_k^\circ(X)$ for all $k \ge 1$.  By Lemma \ref{lem:inductive-open-evaluation-map}, the orbits of the action of $\Aut_k^\circ(X)$ on the closed equivalence class $\pi_{k-1}^{-1}(\pi_{k-1}(x_0))$ are open (noting $\Aut_k^\circ(X)$ is open in $\Aut_k(X)$ as above). Hence these orbits are also closed (as they partition the space), and therefore compact.  If $S$ is the orbit containing $x_0$, then the restricted evaluation map $\ev_{x_0} \colon \Aut_k(X)^\circ \to S$ is continuous, open and surjective and so induces an homeomorphism $\Aut_k^\circ(X) / (\stab(x_0) \cap \Aut_k^\circ(X)) \cong S$, and so the left hand side is compact as required.
\end{proof}

\section{Enough \texorpdfstring{$k$}{k}-translations}
\label{sec-main-details}

We now return to the proof of Proposition \ref{lem:surj}, the ``translation lifting'' statement.

We recall the set-up.  We have a (compact, ergodic) nilspace $X = (X, C^n(X))$ of degree $s$, and are considering the canonical factor map $\pi_{s-1} \colon X \to \pi_{s-1}(X)$.  For notational brevity, we write $\pi$ for $\pi_{s-1}$
and $\pi(X)$ for $\pi_{s-1}(X)$, for the remainder of this section.  We are assuming in particular that the top structure group $A_s(X)$ is a Lie group.

We wish to show the following: for every $\eps > 0$ there exists $\delta > 0$ such that for any $\phi \in \Aut_k(\pi(X))$ with $\|\phi\| \le \delta$, there exists $\tilde{\phi} \in \Aut_k(X)$ with $\|\tilde{\phi}\| \le \eps$ such that $\pi \circ \tilde{\phi} = \phi \circ \pi$.

The strategy is roughly as follows.
\begin{enumerate}[label=(\arabic*)]
  \item We first seek \emph{some function} $\psi \colon X \to X$ such that $\pi \circ \psi = \phi \circ \pi$, and which behaves nicely with respect to the action of $A_s(X)$.  We will also be able to choose $\psi$ to be small.  Crucially, though, it may \emph{not} be a $k$-translation on $X$.
  \item We then argue that we can ``repair'' $\psi$ to a genuine $k$-translation $\tilde{\phi}$.  Since necessarily $\pi \circ \psi = \pi \circ \tilde{\phi}$, this amounts to finding a function $f \colon X \to A_s$ from $X$ to the structure group $A_s$, and setting $\tilde{\phi}(x) = f(x) . \psi(x)$ (where as usual this denotes the action of $A_s(X)$ on $X$).
\end{enumerate}

Here we deviate from the approach taken in \cite{CS12}.
Both proofs rely on the cocycle theory developed in \cite{CS12}
and expounded in Section \ref{sec:cocycle}.
(This is step (2) in the above programme.)
However the lift of small translations are constructed
in different ways. The main difference is that in \cite{CS12} a measurable, but
not necessarily continuous lift is constructed first, while our argument stays in
the continuous category.

We recall that $A_s(X)$ acts freely on $X$, with orbits precisely the fibers of $\pi$.  The standard name for this set-up is that $X \to \pi(X)$ is a \emph{principal bundle} or $A_s$-principal bundle.  However, this will not typically mean that $X$ is homeomorphic as a topological space to $A_s \times \pi(X)$, i.e.~this bundle need not be the trivial bundle.  In our setting, it turns out that this bundle is nonetheless \emph{locally trivial} in the sense that it locally resembles such a direct product.  This result is due to Gleason:

\begin{theorem}[{\cite{G50}*{Theorem 3.3}}]
  \label{gleason-lemma}
  Suppose $A$ is a compact Lie group acting freely on a completely regular topological space $Y$.  Let $\tau \colon Y \to Y / A$ denote the quotient map.  Then for every point $x \in Y / A$ there is a neighbourhood $U$ of $x$ and a local section $\sigma \colon U \to \tau^{-1}(U)$, i.e.~a continuous map such that $\tau \circ \sigma = \id_U$.
\end{theorem}

Although this is a significant result, we remark that the proof simplifies somewhat when the acting group $A$ is abelian, which is the only case we use.

The next result is the key ingredient for part (1) of our argument. It can be extracted from the proof of the first
covering homotopy theorem (\cite{S51}*{Theorem 11.3}), but we include a proof for completeness.

\begin{lemma}
  \label{lem-bundle-lift}
  Let $X$ be a \uppar{compact, ergodic} nilspace of degree $s$ whose structure group $A_s(X)$ is Lie.  For all $\eps > 0$, there exists $\delta > 0$ such that the following holds.  Given any homeomorphism $f \colon \pi(X) \to \pi(X)$ with $\|f\| \le \delta$, there exists a homeomorphism $f' \colon X \to X$ such that:
  \begin{enumerate}
    \item $\|f'\| \le \eps$,
    \item $f'$ is a lift of $f$, i.e.~$\pi \circ f' = f \circ \pi$, and
    \item $f'$ is a \emph{bundle map}, meaning $f'(a . x) = a . f'(x)$ for all $a \in A_s(X)$.
  \end{enumerate}
\end{lemma}
\begin{proof}
  Let $\eps > 0$ be fixed.
  By Theorem \ref{gleason-lemma}, we may choose an open cover $\{U_i\}_{i=1}^m$ of $\pi(X)$ and a family of local sections $\sigma_i \colon U_i \to \pi^{-1}(U_i)$.  We may assume that $\sigma_i$ are uniformly continuous, and so for any fixed $\eta > 0$ (to be specified later) we may refine the cover if necessary so that $d(\sigma_i(x), \sigma_i(x')) \le \eta$ for all $i$ and $x,x' \in U_i$.

  We now choose $\delta$ to be a Lebesgue number for the cover $U_i$; i.e.~for all $x \in \pi(X)$ there is an $i$ such that $\{y \colon d(x,y) \le \delta\} \subseteq U_i$.  Write $U'_i$ for the set of all $x$ such that $\{y \colon d(x,y) \le \delta\} \subseteq U_i$; this is another open cover of $\pi(X)$.

  We would like to define $f'$ as follows.  Given $x \in \pi^{-1}(U'_i)$, there is an unique element $a_i(x) \in A_s(X)$ such that $x = a_i(x) . \sigma_i(\pi(x))$.
(If we use the local section $\sigma_i$ to identify $\pi^{-1}(U'_i)$ with $U_i'\times A_s(X)$,
then $a_i$ is simply the projection to the $A_s(X)$ component.)
Then we set for $x\in U_i'$
  \[
    f'_i(x) = a_i(x) . \sigma_i(f(\pi(x))) \ .
  \]
Note that $d(\pi(x),f(\pi(x)))\le\delta$, hence $f(\pi(x))\in U_i$ for all $x\in U_i'$.

  Note that $d(\sigma_i(\pi(x)), \sigma_i(f(\pi(x)))) \le \eta$ by assumption, and by uniform continuity of the action of $A_s(X)$, we may choose $\eta$ so that
  \[
    d\left(a\, .\, \sigma_i(\pi(x)), a\, .\, \sigma_i(f(\pi(x)))\right) \le \eps' / 2
  \]
  for all $a \in A_s(X)$, where $\eps'$ is some constant to be determined; and so in particular when $a = a_i(x)$ we have
  \[
    d(x, f'_i(x)) \le \eps'/2 \ .
  \]
  Properties (ii) and (iii) for $f'_i$ are clear by construction.

  Unfortunately we are not done: $f'_i$ is only defined on $U'_i$, and the functions for different $i$ are not compatible.  So, we have not constructed a function on all of $X$ with the desired properties.

  To fix this, we will take an average of the values $f_i'(x)$ at each point $x$, in a sense to be made precise.  We first choose a continuous partition of unity $\nu_i \colon \pi(X) \to \RR_{\ge 0}$ adapted to $U'_i$; that is,
  \[
    \sum_{i=1}^m \nu_i(y) = 1
  \]
  for all $y \in \pi(X)$, and $\nu_i$ is supported on $U'_i$ for each $i$.

  We now define
  \[
    f'(x) = \sum_{i=1}^m \nu_i(\pi(x)) f'_i(x) \ .
  \]
  Strictly speaking, this definition makes no sense whatsoever: $f_i'(x)$ is an element of $X$, on which neither addition nor multiplication by real numbers is legitimate.  We now justify what is actually meant by it, as follows.
  \begin{itemize}
    \item For fixed $x$ all the points $\{ f'_i(x) \colon 1 \le i \le m,\, \pi(x) \in U_i \}$ lie in the same fiber of $\pi$, and hence are related to each other by the action of elements of $A_s(X)$.  To fix notation, let $i_0$ be such that $\pi(x) \in U'_{i_0}$, and write $\alpha_i(x) \in A_s(X)$ for the unique element such that $\alpha_i(x) . f'_{i_0}(x) = f'_i(x)$ for each $i$ such that $\pi(x)\in U'_i$.
    \item Since $d(f'_i(x), f'_j(x)) \le \eps'$ for all $i,j$, we find that the elements $\alpha_i(x) \in A_s(X)$ may be assumed to be small with respect to our preferred metric on $A_s(X)$ (by Proposition \ref{prop:robust-free-action}).
    \item Write $A_s(X) \cong (\RR/\ZZ)^d \times K$ for some $d \ge 0$ and some finite group $K$.  The $\alpha_i$ all lie in some small neighbourhood of the identity, which in turn is locally isomorphic to a small ball in $\RR^d \times \{0\}$, for small enough $\eps'$.  Write $\tilde{\alpha_i}(x) \in \RR^d$ for the lift of $\alpha_i(x)$.
    \item It now makes sense to define
      \[
        \tilde{\alpha}(x) = \sum_{i=1}^m \nu_i(x) \tilde{\alpha_i}(x) \in \RR^d
      \]
      which also lies in the same small neighbourhood of the identity; and then to let $\alpha(x) \in (\RR/\ZZ)^d \times \{0\}$ be the projection in $A_s(X)$.  We let
      \[
        f'(x) = \alpha(x) . f'_1(x)
      \]
      for all $x$.
    \item Finally, we remark that this process was independent of the choice of $i_0$, and hence $f'$ defines a function on all of $X$.
  \end{itemize}
  In short, we are doing a kind of integration of functions with values in a compact abelian Lie group $A_s(X)$, or more precisely on the space $\pi^{-1}(f(x))$ which is identified with $A_s(X)$, and claiming this makes sense whenever all the points being averaged are sufficiently close together.   We will need to repeat this kind of process in what follows, and a further discussion of closely related issues appears in Section \ref{subsec:integration}.

  We conclude by arguing that $f'$ as defined has the required properties.  Since $f'(x)$ lies in the same fiber as each $f'_i(x)$ by construction, it is clear that (ii) holds.  Since the elements $\alpha_i(x)$ and hence $\alpha(x)$ are all arbitrarily close to $0 \in A_s(X)$ provided $\eps'$ is sufficiently small, we can ensure that $d(f'_{i_0}(x), \alpha(x). f'_{i_0}(x)) \le \eps/2$ and hence (i) holds, given previous discussion.  It is also not hard to argue that $\alpha_i(x)$ and $\alpha(x)$ are continuous functions of $x$ for each choice of $i_0$, and hence $f'$ is locally continuous and therefore continuous.

  Property (iii) holds because this integration process commutes with translation by $A_s(X)$; indeed, it is clear from the definitions that $\alpha_i(x) = \alpha_i(a.x)$ for any $a \in A_s(X)$.  It is a consequence of (ii) and (iii) that $f'$ is a bijection and hence a homeomorphism.
\end{proof}

We have now completed part (1) of our strategy: we write $\psi \colon X \to X$ for the function $f'$ returned by applying Lemma \ref{lem-bundle-lift} to $\phi$.  To recap, we know that
\begin{itemize}
  \item $\pi \circ \psi = \phi \circ \pi$;
  \item $\|\psi\| \rightarrow 0$ as $\|\phi\| \rightarrow 0$; and
  \item $\psi$ commutes with $A_s(X)$, i.e.~$\psi \circ t_a = t_a \circ \psi$ for all $a \in A_s(X)$.
\end{itemize}

We now wish to attack part (2); that is, ``fixing'' $\psi$ to make it into a $k$-translation.

Just as in the proof of Lemma \ref{stab-discrete}, it is helpful to unwrap what it would mean for a function $\psi$ with these properties to be a $k$-translation on $X$.  More accurately, we want a condition that measures the \emph{failure} of $\psi$ to be a $k$-translation.  The hope is then to construct a correction to this failure, and thereby construct $\tilde{\phi}$.

In the interests of generality, let $\chi \colon X \to X$ denote an arbitrary homeomorphism such that $\pi \circ \chi(x) = \phi \circ \pi(x)$ for all $x \in X$.

By Proposition \ref{prop:k-translation-equiv}, we have that $\chi$ is a $k$-translation if and only if for every $c \in C^{s+1-k}(X)$ we have that $\llcorner^k(c; \chi(c)) \in C^{s+1}(X)$.

By assumption, the projection $\pi(\llcorner^k(c; \chi(c)))$ is a cube of $\pi(X)$.  Hence -- by the weak structure theory -- the obstructions to this configuration being a cube in $X$ can be expressed in terms of the structure group $A_s$.  We recall the following precise statement of this.

\begin{proposition}
  \label{structure-thm-facts}
  Let $c \in C^{s+1}(X)$ be some cube, and $c' \colon \{0,1\}^{s+1} \to X$ a configuration such that $\pi(c) = \pi(c')$.  Let $\alpha \colon \{0,1\}^{s+1} \to A_s(X)$ denote the unique configuration such that $\alpha\, .\, c = c'$.

  Then $c' \in C^{s+1}(X)$ if and only if $\alpha \in C^{s+1}(\cD_s(A_s(X)))$.  This holds if and only if the alternating sum
  \[
    \sum_{\omega \in \{0,1\}^{s+1}} (-1)^{|\omega|} \alpha(\omega)
  \]
  is zero.
\end{proposition}
\begin{proof}
  See \cite{GMV1}*{Theorem 5.4 and Proposition 5.1}. 
\end{proof}

We term this single element of $A_s(X)$ the \emph{discrepancy} of the configuration, which we now define formally.

\begin{definition}
  Let $c' \colon \{0,1\}^{s+1} \to X$ be a configuration such that $\pi(c') \in C^{s+1}(\pi(X))$.  Let $c$ be some element of $C^{s+1}(X)$ such that $\pi(c) = \pi(c')$ (which always exists
  by the definition of quotient cubespaces -- see \cite{GMV1}*{Definition 5.2}). 

  As above, define $\alpha \colon \{0,1\}^{s+1} \to A_s(X)$ to be the unique configuration such that $\alpha \,.\, c = c'$.  Then we define the \emph{discrepancy} $\Delta(c')$ of $c'$ by
  \[
    \Delta(c') = \sum_{\omega \in \{0,1\}^{s+1}} (-1)^{|\omega|} \alpha(\omega) \ .
  \]
\end{definition}

\begin{proposition}
  \label{prop:discrepancy}
  The discrepancy is well-defined, and $\Delta(c') = 0$ if and only if $c' \in C^{s+1}(X)$.
\end{proposition}
\begin{proof}
  The second statement is taken directly from Proposition \ref{structure-thm-facts}.  To verify well-definedness, note that if $c_1, c_2 \in C^{s+1}(X)$ are cubes with $\pi(c_1) = \pi(c_2) = \pi(c')$ and $\alpha_1, \alpha_2 \colon \{0,1\}^{s+1} \to A_s(X)$ satisfy $\alpha_1 \,.\, c_1 = \alpha_2 \,.\,c_2 = c'$, then $(\alpha_1 - \alpha_2). c_1 = c_2$ and invoking the Proposition again we find that
  \[
    \sum_{\omega \in \{0,1\}^{s+1}} (-1)^{|\omega|} (\alpha_1(\omega) - \alpha_2(\omega)) = 0
  \]
  as required.
\end{proof}

We note one more elementary fact about discrepancies.
\begin{proposition}
  \label{prop-discrepancy-additive}
  The discrepancy is \emph{additive} in the following sense: if $c' = [c_0, c_1]$ and $c'' = [c_1, c_2]$ are configurations such that $\pi(c'), \pi(c'') \in C^{s+1}(\pi(X))$ are cubes, then $\pi([c_0,c_2]) \in C^{s+1}(\pi(X))$ and $\Delta([c_0, c_2]) = \Delta(c') + \Delta(c'')$.
\end{proposition}
\begin{proof}
  The first fact is immediate from glueing (see \cite{GMV1}*{Proposition 6.2}). 
  The discrepancy identity is clear from considering the alternating sums.
\end{proof}

We now return to considering our function $\chi$.  With the above discussion in mind, we define a function
\begin{align*}
  \rho_\chi \colon C^{s+1-k}(X) &\to A_s(X) \\
  c &\mapsto \Delta(\llcorner^k(c; \chi(c)))
\end{align*}
i.e.~the discrepancy of the configuration $\llcorner^k(c; \chi(c))$ that arose while characterizing $k$-translations.

Our observations can be summarized by the following result.
\begin{lemma}
  Given $\chi$ as above, we have that $\chi$ is a $k$-translation if and only if $\rho_\chi$ is identically zero.
\end{lemma}
\begin{proof}
  This is immediate from Proposition \ref{prop:k-translation-equiv}, Proposition \ref{prop:discrepancy} and Proposition \ref{structure-thm-facts}.
\end{proof}

So, the failure of $\chi$ to be a $k$-translation is entirely captured by the function $\rho_\chi$.  Now suppose that we take our original function $\psi \colon X \to X$ and attempt to repair it by setting
\[
  \tilde{\phi}(x) = f(x) \,.\, \psi(x)
\]
for some function $f \colon X \to A_s(X)$, as suggested originally.  Then by the lemma, we will have succeeded if and only if $\rho_{\tilde{\phi}} \equiv 0$, which in turn holds if and only if for every $c \in C^{s+1-k}(X)$,
\[
  \rho_\psi(c) + \sum_{\omega \in \{0,1\}^{s+1-k}} (-1)^{|\omega| + k} f(c(\omega)) = 0 \ .
\]

The crucial point is now the following.  The function $\rho_\psi$ is some kind of ``cocycle'': a function on cubes with certain properties that we will describe.  The above equation states that $\rho_\psi$ is some kind of ``coboundary'', i.e.~equal to the ``derivative''
\[
  \pm \sum_{\omega \in \{0,1\}^{s+1-k}} (-1)^{|\omega|} f(c(\omega))  \ .
\]
So, what we want is to show some kind of ``cohomological triviality'': that \emph{every} cocycle $\rho_\psi$ will be a coboundary, and therefore $\psi$ can be corrected to a $k$-translation.

This will not be always true.  However, it \emph{is} true provided $\rho_\psi$ is ``sufficiently small'' in the sense of its image being contained in a small ball in $A_s(X)$.\footnote{An equivalent statement is that the ``cohomology in $\RR^d$'' is trivial, or that the ``cohomology group is discrete'' in an appropriate sense.  We will not make these statements rigorous, though one could certainly do so.}

It remains only to define what we mean by a ``cocycle'' and ``coboundary'', and to state the cohomological triviality result that we will prove in the next section.

\begin{definition}
  \label{def:cocycle}
  Let $\ell \ge 0$ be an integer, $X$ a cubespace and $A$ an abelian group.  By an \emph{$\ell$-cocycle} on $X$ with values in $A$, we mean a continuous function
  \[
    \rho \colon C^\ell(X) \to A
  \]
  satisfying the following conditions:
  \begin{enumerate}
    \item (additivity) if $c' = [c_0, c_1]$, $c'' = [c_1, c_2]$ and $c = [c_0, c_2]$ are all $\ell$-cubes then $\rho(c) = \rho(c') + \rho(c'')$;
    \item (reflections) if $c = [c_0, c_1]$ then $\rho(c) = -\rho([c_1, c_0])$;
    \item (degenerate cubes) if $c = [c_0, c_0]$ then $\rho(c) = 0$.
  \end{enumerate}

  We say $\rho$ is a \emph{coboundary} if there exists a continuous function $f \colon X \to A$ such that
  \[
    \rho(c) = \partial^\ell f(c) := \sum_{\omega \in \{0,1\}^\ell} (-1)^{|\omega|} f(c(\omega)) \ .
  \]
\end{definition}

\begin{remark}\
  \label{rem:cocycle-defs}
  \begin{itemize}
    \item A $0$-cocycle is the same thing as a continuous function $X \to A$.
    \item We stress that our rather vague notation allows the concatenation operation $[-,-]$ to occur on any coordinate $\{1, \dots, \ell\}$, not just the last one.  For instance, for a $2$-cococyle $\rho$ it is the case that
  \[
    \rho([[a,b],[c,d]]) = -\rho([[c,d],[a,b]]) = -\rho([[b,a],[d,c]]) = \rho([[d,c],[b,a]]) \ .
  \]
    \item Note that properties (i), (ii), (iii) are not logically independent, i.e.~our definition is not minimal. Indeed, item (iii) is an immediate consequence of item (i) with the substitution $c_0=c_1=c_2=c_0$. Then item (ii) can also be deduced by taking $c_2=c_0$ in (i) and using (iii).
    \item As well as taking the ``derivative'' $\partial^\ell$ of a function on $X$, one may analogously take the derivative of a function of cubes; e.g.~given $\rho \colon C^n(X) \to A$ we may write
      \[
        \partial \rho([c, c']) = \rho(c) - \rho(c') \ .
      \]
    This definition does now depend implicitly on which coordinate $\{1,\dots,n+1\}$ is used for the concatenation; we will introduce more precise conventions below when this is necessary to avoid confusion.
    \item It is trivial to verify that any coboundary is indeed a cocycle.  More generally, if $\rho$ is an $\ell$-cocycle then $\partial^k \rho$ is an $(\ell+k)$-cocycle.
  \end{itemize}
\end{remark}

\begin{lemma}
  For any $\chi$ having the properties discussed above, the function $\rho_\chi$ is indeed a cocycle.
\end{lemma}
\begin{proof}
  Additivity follows directly from the additivity of discrepancies, proven in Proposition \ref{prop-discrepancy-additive}.  The other two properties follow from this.
\end{proof}

In summary, at long last, it will suffice to prove the following fact, which is a special case of Theorem \ref{thm:cocycle-main}.

\begin{theorem} \cite{CS12}*{Lemma 3.19}
  \label{thm:cocycle-special}
  Suppose $X$ is a compact ergodic nilspace of degree $s$, $A$ is a compact abelian Lie group
(equipped with a metric $d_A$),
$\ell \ge 0$ is an integer and $\rho \colon C^\ell(X) \to A$ is an $\ell$-cocycle.

  Then there exists $\eps = \eps(s, \ell, A) > 0$, depending only on $s$, $\ell$ and $A$ \uppar{but not on $X$}, such that the following holds.  Suppose that $\delta \le \eps$ and $d_A(\rho(c), \rho(c')) \le \delta$ for all $c, c' \in C^\ell(X)$.  Then there exists $f \colon X \to A$ such that $\rho = \partial^\ell f$ and also $d_A(f(x), f(x')) \lesssim_{s,\ell} \delta$ for all $x, x' \in X$.
\end{theorem}

We now bring everything together to conclude the proof of Proposition \ref{lem:surj}. If the original map $\phi$ has $\|\phi\|$ small enough, then $\|\psi\|$ is also arbitrarily small.  Hence, the cocycle $\rho_\psi$  can be taken to be as small as we like with respect to the metric
\[
  \sup_{c, c' \in C^\ell(X)} d_{A_s(X)}(\rho_\psi(c), \rho_\psi(c')) \ ;
\]
indeed, if $\psi$ is small then $\llcorner^k(c; \psi(c))$ is close to a genuine cube $\square^k(c)$, and since the discrepancy map $\Delta$ is continuous (where here we have used Proposition \ref{prop:robust-free-action}) it follows that $\rho_\psi(c)$ is close to $0$ for every $c$.

Hence, the correction $f$ from Theorem \ref{thm:cocycle-special} can be chosen to be arbitrarily small, meaning the final $k$-translation $\tilde{\phi} \in \Aut_k(X)$ that we obtain is also arbitrarily small, as required.

\section{Cocycle theory}
\label{sec:cocycle}

The purpose of this section is to prove Theorem \ref{thm:cocycle-special} concerning triviality of small cocycles.

It will be necessary elsewhere in this project to have a version of this result in greater generality than as stated therein.  The proof of this more general version is very similar to a direct proof of the specialized version; the
only difference is that we need to draw on the \emph{relative} weak structure theory, as expounded in \cite{GMV1}*{Section 7}, 
which generalizes (again in a fairly routine way) the absolute version from \cite{CS12} (or \cite{GMV1}*{Section 6}).

The technical generalization is as follows.
The reader is advised to recall the notion of a fibration, and other related terminology, from
\cite{GMV1}*{Section 7}. 

\begin{theorem}
  \label{thm:cocycle-main}

  Let $A$ be a compact abelian Lie group (equipped with a metric $d_A$), and let $s \ge 0$, $\ell \ge 1$ be given.  Then there exists $\eps = \eps(s, \ell, A) > 0$ such that the following holds.

  Let $\beta \colon X \to Y$ be any fibration of degree $s$ between compact ergodic cubespaces $X$ and $Y$ that obey the glueing axiom.

  Now let $\rho$ be a continuous $\ell$-cocycle on $X$ with values in $A$, let $0 < \delta \le \eps$ be given and suppose that $d(\rho(c), \rho(c')) \le \delta$ whenever $\beta(c) = \beta(c')$.

  Then $\rho = \partial^\ell f + \tilde{\rho} \circ \beta$, where $f \colon X \to A$ is continuous and satisfies $d(f(x), f(y)) \lesssim_{s,\ell} \delta$ \uppar{that is, there exists a constant $c=c(s,\ell)>0$ such that $d(f(x), f(y))\leq c\delta$} whenever $\beta(x) = \beta(y)$, and $\tilde{\rho}$ is a continuous $\ell$-cocycle on $Y$.
\end{theorem}

(Recall $\partial^\ell$ was defined in Definition \ref{def:cocycle}.)

We briefly pause to explain why this does in fact generalize Theorem \ref{thm:cocycle-special}.
  Letting $X$, $A$ and $\rho$ be as in the statement of Theorem \ref{thm:cocycle-special}, we may invoke Theorem \ref{thm:cocycle-main} where:
  \begin{itemize}
    \item $X$ is still the same (compact, ergodic, degree $s$) nilspace $X$;
    \item $Y$ is $\{\ast\}$, the one-point space;
    \item $\beta \colon X \to \{\ast\}$ is the trivial map; and
    \item the cocycle $\rho$ is unchanged.
  \end{itemize}
  The statement that $\beta$ is a fibration of finite degree is precisely saying that $X$ is a nilspace of finite degree.  Hence, all the hypotheses of Theorem \ref{thm:cocycle-main} in this special case follow from those of Theorem \ref{thm:cocycle-special}.  Given the conclusion of Theorem \ref{thm:cocycle-main}, note that the only $\ell$-cocycle $C^\ell(\{\ast\}) \to A$ is the zero function, and so $\tilde{\rho} = 0$ and we recover the conclusion of Theorem \ref{thm:cocycle-special}.

Even though Theorem \ref{thm:cocycle-main} does indeed generalize Theorem \ref{thm:cocycle-special}, since its proof will involve some additional notational and technical difficulties, we will continue to provide sketches of the argument in the special case of Theorem \ref{thm:cocycle-special}.  The hope is that the intuition will be easier to grasp in this model setting.

\bigskip

Note that, in Section \ref{sec:main-outline}, we needed a kind of dual to Theorem \ref{thm:cocycle-special}, stating that the function $f \colon X \to A$ that is used to construct the coboundary $\partial^\ell f$ is itself essentially uniquely determined.  The precise statement is the following.

\begin{theorem}
  \label{thm:cocycle-uniqueness}
  Let $A$ be a compact abelian Lie group, equipped with the metric $d_A$, and let $s \ge 0$, $\ell \ge 1$ be given. Then there exists $\eps = \eps(s, \ell, A) > 0$ such that the following holds.

  Suppose $X = (X, C^n(X))$ is a compact ergodic nilspace of degree $s$, and $f \colon X \to A$ is a continuous function such that $d_A(f(x), f(x')) \le \eps$ for all $x, x' \in X$, and $\partial^\ell f \colon C^\ell(X) \to A$ is the zero function.

  Then $f$ is constant.
\end{theorem}

Again, we would like a version that holds in the setting of Theorem \ref{thm:cocycle-main}, concluding that the function $f \colon X \to A$ obtained there is essentially unique.  It turns out that in this case, this generalization follows easily from the original version.

\begin{corollary}
  \label{cor:general-uniqueness}
  Suppose $A$, $\beta$, $X$, $Y$ are as in the statement of Theorem \ref{thm:cocycle-main}.  Then there exists $\eps = \eps(A) > 0$ such that the following holds.

  Suppose $\ell \ge 0$ is an integer, and $f \colon X\to A$ is a continuous function such that $d_A(f(x), f(x')) \le \eps$ for all $x, x' \in X$ and furthermore $\partial^\ell f$ is constant on fibers of $\beta$ \uppar{meaning if $c, c' \in C^\ell(X)$ and $\beta(c) = \beta(c')$ then $\partial^\ell f(c) = \partial^\ell f(c')$}.

  Then $f$ is also constant on fibers of $\beta$, i.e.~$f = \tilde{f} \circ \beta$ for some continuous $\tilde{f} \colon Y \to A$.
\end{corollary}

\begin{proof}[{Proof of Corollary \ref{cor:general-uniqueness} from Theorem \ref{thm:cocycle-uniqueness}}]
  Given an $y \in Y$, its preimage $Z := f^{-1}(y) \subseteq X$ is a nilspace of degree $s$.  Moreover, by hypothesis $\partial^\ell f$ is constant on $C^\ell (Z)$, and therefore zero on $C^\ell(Z)$, since clearly $\partial^\ell f(\square^\ell(z)) = 0$ for any $z \in Z$.

  Applying Theorem \ref{thm:cocycle-uniqueness}, we deduce that $f$ is constant on $Z$.  But since $y$ was arbitrary, this suffices.
\end{proof}

We will spend a while building up the proof of Theorem \ref{thm:cocycle-main} in stages.  The uniqueness result (Theorem \ref{thm:cocycle-uniqueness}) is somewhat easier, and we return to this at the end.

\subsection{The case \texorpdfstring{$\ell = 1$}{l=1}}
\label{subsec:1-cocycle}

We briefly contemplate the first even slightly non-trivial case of Theorem \ref{thm:cocycle-special}, i.e.~when $\ell = 1$.  Since $X$ is ergodic, $C^1(X) = X \times X$, and so this means we have a function
\[
  \rho \colon X \times X \to A
\]
whose image is contained in a small ball of $A$, and such that
\[
  \rho([x,y]) + \rho([y,z]) = \rho([x,z])
\]
for all $x,y,z \in X$ (this is just from the definition of a $1$-cocycle).  Also, $\rho([x,y]) = -\rho([y,x])$ and $\rho([x,x]) = 0$ for all $x,y$.

Finally, recall we wish to construct $f \colon X \to A$ such that $\rho([x,y]) = f(x) - f(y)$ for all $x,y \in X$.  It turns out this is not a terribly deep statement: one can simply fix $x_0 \in X$ and define
\[
  f(x) := \rho([x, x_0])
\]
noting that
\[
  f(x) - f(y) = \rho([x, x_0]) - \rho([y, x_0]) = \rho([x, x_0]) + \rho([x_0, y]) = \rho([x,y]) \ .
\]
Observe that we have not even really used our ``small image'' assumption on $\rho$, or that $A$ is Lie, in this argument.

\bigskip

When we come to prove more general cases of the result, it will be important for inductive reasons that we can find a \emph{canonical} function $f$ to define our coboundary: here, $f$ depends on the arbitrary choice of $x_0$. To set the scene for what follows, we note that we can achieve this if, rather than fixing one particular $x_0$, we instead \emph{average} over all choices.  That is, we define
\begin{equation}\label{eq:f-def}
  f(x) := \int_X \rho([x,y]) dy
\end{equation}
which is indeed canonical.  To show this is still a valid choice, we work through the calculation
\begin{align*}
  f(x) - f(y) &= \int_X \rho([x, z])\ dz - \int_X \rho([y, z])\ dz \\
              &= \int_X \left[ \rho([x, z]) + \rho([z, y]) \right]\ dz\\
              &= \int_X \rho([x,y])\ dz\\
              &= \rho([x,y]) \ .
\end{align*}
However, we have overlooked two important subtleties here.
\begin{itemize}
  \item We are integrating expressions with values in $A$, a compact abelian Lie group.
  \item We have not defined a probability measure on the space $X$, so the integration is currently meaningless.
\end{itemize}

The first item will be discussed in Section \ref{subsec:integration} below.
In brief, we can lift the integrand to the universal cover of (the respective connected
component of) $A$, and then project back the result.
It turns out that this can be defined canonically provided the integrand takes values in a sufficiently
small ball in $A$.

The second item could be addressed using the weak structure theory, which implies that $X$ can
be obtained by successive principal bundle extensions
\[
  X \to \pi_{s-1}(X) \to \dots \to \pi_{0}(X) = \{\ast\}.
\]
This allows the construction of a probability measure on $X$ from the Haar measures on the structure groups.
This is the approach taken in \cite{CS12}.

We found it easier to carry out the averaging in stages and prove Theorem \ref{thm:cocycle-main}
by a double induction on $s$ and $\ell$.
In the special case $\ell=1$ currently discussed, this takes the following shape.
We put
\[
f'=\int_{A_s} \rho([x,a.x])d\mu_{A_s}(a),
\]
i.e. we integrate only on the fibre of $\pi_s$ that contains $x$ using the Haar measure on the structure group $A_s$.

It is no longer reasonable to expect that $\partial f'$ agrees with $\rho$, since the definition of $f'$ depends
only on the values of $\rho$ on edges whose endpoints lie in the same fibre of $\pi_{s-1}$.
However, it turns out that $\rho-\partial f'$ is constant on the fibres of $\pi_{s-1}$ and hence it
can be pushed down to a cocycle on $\pi_{s-1}(X)$, and the argument can be completed by induction.
We omit the details, but they are given in Section \ref{subsec:general} in a greater generality.

We finish this section with a comment on the proof in the general case $\ell\ge 1$. Antol\'\i n Camarena and Szegedy define the function
\[
f(x)=\int_{c\in C^{\ell}(X):c(\vec0)=x}\rho(c) dc
\]
using a system of suitable probability measures that are defined on the space of cubes whose vertex
at $\vec 0$ is the point $x$. (Compare this formula with \eqref{eq:f-def}.)

Our proof carries out the above averaging in several steps using a double induction on $s$ and $\ell$. If one combines the averaging in the inductive steps, the resulting formula will be the same as in the approach of Antol\'\i n Camarena and Szegedy. However, we believe that the technical details of the argument become simpler in our approach. In particular, we can avoid any  discussion of continuous systems of measures, and moreover the combinatorial properties we need to verify are also easier.

\subsection{A remark on integration}
\label{subsec:integration}

We discuss now the issue raised above about integration of functions taking values in Lie groups.

Suppose $(Y, \mu)$ is some probability space, and $f \colon Y \to \RR / \ZZ$ is some measurable function.  It is  clear that the integral $\int f(x) d\mu(x)$ does not make sense in general, since averaging on $\RR/\ZZ$ is not well-defined.

However, suppose we know that $f$ takes values in some specified interval $(a,b) \subseteq \RR/\ZZ$ of width at most $1/10$ (say).  Then we can make sense of $\int f(x) d\mu(x)$ by identifying $(a,b) \subseteq \RR/\ZZ$ bijectively with a suitable interval $(\tilde{a}, \tilde{b}) \subseteq \RR$, lifting $f$ to a function $Y \to (\tilde{a}, \tilde{b})$, performing normal real integration, and then projecting the result (which lies in $(\tilde{a}, \tilde{b})$ by convexity) back down to $(a,b) \subseteq \RR/\ZZ$.  It is straightforward to check that this operation does not depend on the choice of $\tilde{a}, \tilde{b}$.

We will abuse notation to write $\int f(x) d\mu(x)$ for this element of $\RR/\ZZ$, whenever it makes sense to do so.  Specifically, we will only write this when $\sup \{|f(x) - f(y)| \colon x, y \in Y \} \le \delta$ for some suitably small absolute constant $\delta$ (in this case taken to be $1/10$).

More generally, the same remarks hold for any compact abelian Lie group $A = (\RR/\ZZ)^d \times K$ for $K$ a finite group.  For any reasonable metric on this space, integration is well-defined for functions whose image lies in a sufficiently small $\delta$-ball, with $\delta$ depending only on $A$; and the integral will also lie in the same $\delta$-ball.

For this to hold, we do need to make some choices about the metric we impose on $A$.  Specifically, it will be convenient to use a metric $d_A$ on $A \cong (\RR/\ZZ)^d \times K$ that is induced from the standard Euclidean metric on $\RR^d$ (and the standard discrete metric on $K$).  This choice is not canonical, because it requires us to fix an isomorphism $A \cong (\RR/\ZZ)^d \times K$; though of course, all metrics on a compact space are equivalent, so we can simply fix a choice and the ambiguity need not concern us.

\begin{definition}
  \label{def:torus-metric}
  For any compact abelian Lie group $A$ that appears, we assume an identification $A \cong (\RR/\ZZ)^d \times K$ has been fixed and define
  \[
    d_A((t, k), (t',k')) = (1 - [k=k']) + \inf \left\{ \|x - x'\|_2 \colon x, x' \in \RR^d,\ x, x' \equiv t, t' \pmod{1} \right\} \ .
  \]
  Similarly, we write $\|\cdot\|_A$ for the ``norm'' $d_A(0, -)$.
\end{definition}

(Here the expression $[k=k']$ takes the value $1$ if $k=k'$ and $0$ otherwise.)

Finally, we observe that the integral discussed above is finitely additive wherever this makes sense; i.e.~if $f$, $f'$ and $f+f'$ all have images within (possibly different) $\delta$-balls in $A$, then $\int (f+f') = \int f + \int f'$.

\subsection{A slightly less easy case}
\label{subsec:ds-cocycle-fixing}
As one further stepping stone towards Theorem \ref{thm:cocycle-special} and Theorem \ref{thm:cocycle-main}, we consider the case where $X = \cD_s(H)$ for some compact abelian group $H$.  Recall this means $H$ carries the degree $s$ filtration
\[
  H = H_0 = H_1 = \dots = H_s \supseteq \{0\}
\]
and $C^k(X)$ are the Host--Kra cubes $\HK^k(H_\bullet)$ with respect to this filtration.  (We use additive notation for $H$ and also for the Host--Kra groups we derive from it.)

We prove Theorem \ref{thm:cocycle-special} in the special case considered in this section by induction on $\ell$. In the inductive step, we are looking for an $(\ell-1)$-cocycle $\rho': C^{\ell-1}\to A$ such that $\partial\rho'=\rho$. Motivated by the previously discussed case of $\ell=1$, we intend to define $\rho'(c)$ by taking an average of $\rho([c,c'])$ over the space
\[
  \left\{ c' \colon [c, c'] \in C^\ell(X) \right\}.
\]

In order to compute the average, we need to specify a probability measure. Luckily, the space can be identified with a compact abelian Lie group, and we can make use of the Haar measure on this group.  We denote
\[
  T_1^\ell := \left\{ c \colon \{0,1\}^{\ell-1} \to H \ \colon \ [\vec{0}, c] \in C^\ell(X) \right\}
\]
and observe that $[\vec{0}, T_1^\ell]$ is a closed subgroup of $\HK^\ell(H_\bullet)$.  It is clear that $[c_1, c_2] \in C^\ell(\cD_s(H))$ if and only if $c_2 = c_1 + t$ for some $t \in T_1^\ell$: indeed, the Host--Kra cube group is a group, so since $[c_1, c_2]$ and $[c_1, c_1]$ are both cubes, so is $[c_1, c_2] - [c_1, c_1] = [\vec{0}, c_2 - c_1]$ and so $t=c_2 - c_1$ is in $T_1^\ell$.

In fact, it turns out that $T_1^\ell$ is just the Host--Kra cubespace of dimension $(\ell-1)$ of $H$ given the degree $(s-1)$ filtration, but we will not actually need to know this.

The special case of Theorem \ref{thm:cocycle-special} considered in this section follows from the following proposition by induction on $\ell$.

\begin{proposition}
  \label{prop:ds-cocycle-fixing}
  Suppose $X = \cD_s(H)$ as above, $A$ is a compact abelian Lie group, and $\ell \ge 1$ and $\rho \colon C^\ell(X) \to A$ is a cocycle such that $d_A(\rho(c_1), \rho(c_2)) \le \delta$ for all $c_1$, $c_2$ and some suitably small $\delta$ \uppar{i.e.~$\delta \le \eps(\ell, A)$}.

  We define
  \begin{align*}
    \rho' \colon C^{\ell-1}(X) &\to A \\
    c' &\mapsto \int_{T_1^\ell} \rho([c', c'+t]) d\mu_{T_1^\ell}(t) \ .
  \end{align*}
  Then
  \begin{enumerate}
    \item the function $\rho'$ is continuous and $d_A(\rho'(c_1'), \rho'(c_2')) \le \delta$ for all $c_1', c_2' \in C^{\ell-1}(X)$;
    \item moreover, $\rho'$ is an $(\ell-1)$-cocycle; and
    \item we have $\rho([c_1, c_2]) = \rho'(c_1) - \rho'(c_2)$ for all cubes $c = [c_1, c_2] \in C^\ell(X)$.
  \end{enumerate}
\end{proposition}

\begin{proof}
 Note we are drawing on Section \ref{subsec:integration} to make sense of the integral in the definition of $\rho'$.  To see (i), we note that
  \[
    \rho'(c_1') - \rho'(c_2') = \int_{T_1^\ell} (\rho([c_1', c_1' + t]) - \rho([c_2', c_2'+t])) d \mu_{T_1^\ell}(t)
  \]
  and since $\|\rho(c_1) - \rho(c_2)\|_A \le \delta$ for all $c_1, c_2 \in C^\ell(X)$ by assumption, the integrand is bounded in norm pointwise by $\delta$ and hence so is the integral.
  Moreover, since the maps
  \[
    c' \mapsto [c', c'+t]
  \]
  are equicontinuous for $t \in T_1^\ell$, and since $\rho$ is uniformly continuous, we deduce that the integrand becomes arbitrarily small when $d(c_1', c_2')$ is arbitrarily small. Again, an average of small values in $A$ is small, so $\rho'$ is (uniformly) continuous.

  We now consider (ii).  We need some further definitions.  Let
  \[
    T_2^\ell := \left\{ c \colon \{0,1\}^{\ell-2} \to H \ \colon \ [[\vec{0}, \vec{0}], [\vec{0},c]] \in C^\ell(H) \right\}
  \]
  and note as before that $[t_0, t_1] \in T_1^\ell$ if and only if $t_1 = t_0 + u$ for some $u \in T_2^\ell$ (again, just using the fact that the Host--Kra cube group is a group).  In other words, every $t \in T_1^\ell$ has a unique decomposition as
  \[
    t = [v, v] + [\vec{0}, u]
  \]
  for $v \in T_1^{\ell-1}$ and $u \in T_2^\ell$, and this establishes a group isomorphism $T_1^\ell \leftrightarrow T_1^{\ell-1} \times T_2^{\ell}$. (Indeed, it is easy to see that $[v,v]\in T_1^\ell$ if and only if $[[0,0],[v,v]]\in C^{\ell}(H)$ if and only if $[0,v]\in C^{\ell-1}(H)$ if and only if $v\in T_1^{\ell-1}$.)

  To prove additivity of $\rho'$, suppose $[c_0, c_1], [c_1, c_2] \in C^{\ell-1}(X)$, and write
  \begin{align*}
    \rho'([c_0, c_2]) &= \int_{T_1^\ell} \rho([[c_0, c_2], [c_0, c_2] + t]) d\mu_{T_1^\ell}(t) \\
                      &= \int_{T_1^{\ell-1} \times T_2^\ell} \rho([[c_0, c_2], [c_0 + v, c_2 + u + v]]) d\mu_{T_1^{\ell-1}}(v) d\mu_{T_2^\ell}(u) \\
                      &= \int_{T_1^{\ell-1} \times T_2^{\ell}\times T_2^\ell} (\rho([[c_0, c_1], [c_0 + v, c_1 + u' + v]]) \\
&\qquad\qquad\qquad\qquad
+ \rho([[c_1, c_2], [c_1 + u'+ v, c_2 + u + v]]) )  d\mu_{T_1^{\ell-1}}(v) d\mu_{T_2^\ell}(u) d\mu_{T_2^\ell}(u')\\
\end{align*}
In the last line we used additivity of $\rho$ and that the cubes in the last line can be glued along the common face $[c_1,c_1+u'+v]$ to obtain the cube in the penultimate line.

We integrate out the first term and continue the calculation as follows.
\begin{align*}
    \rho'([c_0, c_2]) &= \rho'([c_0, c_1]) + \int_{T_2^\ell}\Big(\int_{T_1^{\ell-1}\times T_2^\ell}\rho([[c_1, c_2], [c_1 +  (v+u'),\\
&\qquad\qquad\qquad\qquad\qquad\qquad\qquad
 c_2 + (u - u')+ (v+u')]])  d\mu_{T_1^{\ell-1}}(v) d\mu_{T_2^\ell}(u)\Big) d\mu_{T_2^\ell}(u')\\
		&= \rho'([c_0, c_1]) + \int_{T_2^\ell}\Big(\int_{T_1^{\ell-1}\times T_2^\ell}\rho([[c_1, c_2], [c_1 +  v'', c_2 + u''+ v'']])  d\mu_{T_1^{\ell-1}}(v'') d\mu_{T_2^\ell}(u'')\Big) d\mu_{T_2^\ell}(u')\\
		&= \rho'([c_0, c_1])+\rho'([c_1, c_2])
  \end{align*}
  as required.  For the penultimate equation we used the substitution $v''=v+u'$ and $u''=u-u'$ and translation invariance of the Haar measure.

  We finish the proof with (iii).  Given $[c_1, c_2] \in C^\ell(X)$ we may write $c_2 = c_1 + t$ for some $t \in T_1^\ell$ and then observe
  \begin{align*}
    \rho([c_1, c_2]) &= \int_{T_1^\ell} \left(\rho([c_1, c_1+w]) + \rho([c_1+w, c_1+t] )\right) d\mu_{T_1^\ell}(w) \\
                     &= \int_{T_1^\ell} \rho([c_1, c_1+w]) d\mu_{T_1^\ell}(w) - \int_{T_1^\ell} \rho([c_1+t, c_1+w]) d\mu_{T_1^\ell}(w)
  \end{align*}
  which once again is just $\rho'(c_1) - \rho'(c_2)$ after reparameterization of the last integral (using that $t \in T_1^\ell$ and Haar measure is translation-invariant).
\end{proof}

\subsection{The general case}\label{subsec:general}

The key ideas for the general case, Theorem \ref{thm:cocycle-main}, have already appeared in the special cases we have just discussed.  The very vague strategy they suggest for proving the general case is as follows.
\begin{itemize}
  \item As in Section \ref{subsec:ds-cocycle-fixing}, we will perform an averaging over a space of cubes
    \[
      \{c' \colon [c, c'] \in C^\ell(X) \}
    \]
    to obtain an $(\ell-1)$-cocycle $\rho'$ from an $\ell$-cocycle $\rho$, such that $\rho([c_1, c_2]) = \rho'(c_1) - \rho'(c_2)$.  Iterating this argument $\ell$ times will give the result.
  \item Unfortunately we do not currently have measures defined on these spaces, so we will do as is suggested in Section \ref{subsec:1-cocycle} and appeal to the weak structure theory.  This will enable us to express the averaging as a sequence of integrals over each of the structure groups $A_s(X)$ in turn, or more accurately over Host--Kra configurations built out of these groups.
\end{itemize}

With some thought, this approach is seen to be essentially equivalent to performing a \emph{double} induction on both $\ell$ (the order of the cocycle) and $s$ (the degree of the nilspace $X$ in Theorem \ref{thm:cocycle-special}, or of the fibration $\beta$ in the case of Theorem \ref{thm:cocycle-main}).  In each stage of the induction, we perform some integration over groups related to $A_s(X)$ in the spirit of the previous subsections, and are left at the end with some simpler objects on which to iterate.

We need to introduce some notation for technical reasons that will become clear later. Let $c_1,c_2:\{0,1\}^\ell\to X$ be two configurations. Until now, we defined the concatenation by the same symbol irrespective of the coordinate on which the concatenation takes place. Now we wish to designate this in our notation, and write:
\begin{align*}
[c_1,c_2]_k:\{0,1\}^{\ell+1}\to& X\\
\omega\mapsto&
\begin{cases}
c_1(\omega_1,\ldots,\omega_{k-1},\omega_{k+1},\ldots,\omega_{\ell+1}) & \text{if $\omega_k=0$}\\
c_2(\omega_1,\ldots,\omega_{k-1},\omega_{k+1},\ldots,\omega_{\ell+1}) & \text{if $\omega_k=1$}.
\end{cases}
\end{align*}

We also introduce some notation for the derivative along a specific coordinate.
If $\rho:C^{\ell}\to A$ is a function on cubes, we write
\[
\partial_k\rho([c_1,c_2]_k)=\rho(c_1)-\rho(c_2).
\]
Observe that the identity
\[
\partial^{\ell+1} f=\partial_k (\partial^{\ell} f)
\]
holds irrespective of the value of $k$.
Hence, if we are differentiating a coboundary $\partial^{\ell} f$, the derivative is independent of the direction $k$; however, this will not be the case in general.

Armed with this notation, we can now formulate the inductive step in the proof
of Theorem \ref{thm:cocycle-special}.

\begin{lemma}
  \label{lem:inductive-cocycle-special}
  Let $A$ be a compact abelian Lie group with the metric $d_A$ as above, and let $s \ge 1$, $\ell \ge 1$ be given.  Then there exists $\eps = \eps(s, \ell, A)$ such that the following holds.

  Suppose $X$ is a \uppar{compact, ergodic} nilspace of degree $s$, and $\rho \colon C^\ell(X) \to A$ is an $\ell$-cocycle such that $d_A(\rho(c), \rho(c')) \le \delta$ for all $c,c' \in C^\ell(X)$ and some $\delta \le \eps$.

  Write $\pi_{s-1} \colon X \to \pi_{s-1}(X)$ for the canonical factor.  Then we may decompose
  \[
    \rho(c) = \sum_{k=1}^{\ell}\partial_k\rho'_k(c) + \tilde{\rho}(\pi_{s-1}(c))
  \]
  where
  \begin{itemize}
    \item $\rho'_k \colon C^{\ell-1}(X) \to A$ is an $(\ell-1)$-cocycle for each $k$, and
    \item $\tilde{\rho} \colon C^\ell(\pi_{s-1}(X)) \to A$ is an $\ell$-cocycle on $\pi_{s-1}(X)$,
  \end{itemize}
  such that both $\rho'_k$ and $\tilde{\rho}$ take images in a small ball in $A$, i.e.~$d_A(\rho'_k(c_1), \rho'_k(c_2)) \lesssim_{s, \ell} \delta$ for all $c_1, c_2 \in C^{\ell-1}(X)$ and $k$, and $d_A(\tilde{\rho}(c_1), \tilde{\rho}(c_2)) \lesssim_{s,\ell} \delta$ for all $c_1, c_2 \in C^{\ell}(\pi_{s-1}(X))$.
\end{lemma}

We verify that this is enough to prove Theorem \ref{thm:cocycle-special}.
\begin{proof}[Proof of Theorem \ref{thm:cocycle-special} assuming Lemma \ref{lem:inductive-cocycle-special}]
  We remark again that an ergodic nilspace of degree $0$ is just the one-point space $\{\ast\}$, and so a cocycle of any positive order on this space is identically zero.   Also recall that a $0$-cocycle is just the same thing as a continuous function.

  So, we may proceed by induction on $s$ and $\ell$, where the cases $\ell = 0$ and $s=0$ are both clear.  Given $s, \ell > 0$ and a cocycle $\rho \colon C^\ell(X) \to A$ as in Theorem \ref{thm:cocycle-special}, we decompose
   \[
    \rho(c) = \sum_{k=1}^{\ell}\partial_k\rho'_k(c) + \tilde{\rho}(\pi_{s-1}(c))
  \]
as in Lemma \ref{lem:inductive-cocycle-special}.

  Temporarily we will say that a function $f \colon Y \to A$ has \emph{small image} if $d_A(f(y), f(y')) \lesssim_{s, \ell} \delta$ for all $y, y' \in X$.  By inductive hypothesis, $\rho'_k = \partial^{\ell-1} g_k$ for some continuous $g_k \colon X \to A$ with small image, and similarly $\tilde{\rho} = \partial^\ell h$ for some continuous $h \colon \pi_{s-1}(X) \to A$ with small image.

  Setting $f = \sum_{k=1}^{\ell} g_k + h \circ \pi_{s-1}$, we see that this is again a continuous function $X \to A$ with small image, and moreover
  \[
    \partial^\ell f = \sum_{k=1}^\ell \partial^\ell g_k + \partial^\ell(h \circ \pi_{s-1})
                    = \sum_{k=1}^\ell\partial_k \rho'_k + \tilde{\rho} = \rho
  \]
  which completes the proof.
\end{proof}

Lemma \ref{lem:inductive-cocycle-special} is proved by iterating the following
result. Recall from Section \ref{subsec:ds-cocycle-fixing}, that  we write $T_{1}^\ell$ for the
set of configurations $t:\{0,1\}^{\ell-1}\to A_s(X)$ that satisfy
\[
[0,t]\in C^\ell(\cD_s(A_s)).
\]
(Here we did not indicate in which coordinate the concatenation takes place, as the resulting group is independent of this choice.)

\begin{lemma}
  \label{lem:inductive-cocycle-more-special}
  Let $A$ be a compact abelian Lie group with the metric $d_A$ as above, and let $s \ge 1$, $\ell \ge 1$ be given.  Then there exists $\eps = \eps(s, \ell, A)$ such that the following holds.

  Suppose $X$ is a \uppar{compact, ergodic} nilspace of degree $s$, and $\rho \colon C^\ell(X) \to A$ is an $\ell$-cocycle such that $d_A(\rho(c), \rho(c')) \le \delta$ for all $c,c' \in C^\ell(X)$ and some $\delta \le \eps$.

  Write $\pi_{s-1} \colon X \to \pi_{s-1}(X)$ for the canonical factor. Fix a number $1\le k\le \ell$. Then we may decompose
  \[
    \rho(c) = \partial_k\rho'(c) + \tilde{\rho}(c)
  \]
  where
  \begin{itemize}
    \item $\rho' \colon C^{\ell-1}(X) \to A$ is an $(\ell-1)$-cocycle, and
    \item $\tilde{\rho} \colon C^\ell(X) \to A$ is an $\ell$-cocycle, which is invariant under the action of the group $[0,T_{1}^\ell]_k$,
  \end{itemize}
  such that both $\rho'$ and $\tilde{\rho}$ take images in a small ball in $A$, i.e.~$d_A(\rho'(c_1), \rho'(c_2)) \lesssim_{s, \ell} \delta$ for all $c_1, c_2 \in C^{\ell-1}(X)$ and $d_A(\tilde{\rho}(c_1), \tilde{\rho}(c_2)) \lesssim_{s,\ell} \delta$ for all $c_1, c_2 \in C^{\ell}(X)$.
\end{lemma}

The proof of Lemma \ref{lem:inductive-cocycle-more-special}  follows Section \ref{subsec:ds-cocycle-fixing} very closely, which is possible by virtue of the weak structure theory.  There will be some additional complications caused by the ``relative'' nature of Lemma \ref{lem:inductive-cocycle-more-special} as compared to Proposition \ref{prop:ds-cocycle-fixing}.

\begin{proof}[Proof of Lemma \ref{lem:inductive-cocycle-more-special}]
  The definition of $\rho'$ will be  similar to that in Proposition \ref{prop:ds-cocycle-fixing}.  Instead of integrating over the set
  \[
    \left\{ c' \colon [c, c']_k \in C^\ell(X) \right\}
  \]
  we will integrate over all such $c'$ such that additionally $c, c'$ lie in the same fiber of $\pi_{s-1}$.  By the weak structure theory, this is equivalent to saying
  \[
    \rho'(c) = \int_{T_1^\ell} \rho([c, t.c]_k) d\mu_{T_{1}^\ell}(t)
  \]
  where as always $t.c$ denotes the action of $A_s(X)$ on $X$, applied pointwise to the configurations $t \colon \{0,1\}^{\ell-1} \to A_s(X)$, $c \colon \{0,1\}^{\ell-1} \to X$.

  It suffices to check that
  \begin{enumerate}
    \item the function $\rho' \colon C^{\ell-1}(X) \to A$ is continuous and has small image (in the above sense);
    \item also, $\rho'$ is an $(\ell-1)$-cocycle on $X$; and
    \item the function
      \[
        \tilde{\rho}([c_1, c_2]_k) := \rho([c_1, c_2]_k) - (\rho'(c_1) - \rho'(c_2))
      \]
      is a continuous $\ell$-cocycle on $X$ with small image, and furthermore is invariant under the action of $[0,T_{1}^{\ell}]_k$.
  \end{enumerate}
 Part (i) follows by the same arguments as in Proposition \ref{prop:ds-cocycle-fixing}.  Furthermore, the proof of (ii) is identical to the corresponding part of Proposition \ref{prop:ds-cocycle-fixing}, where we simply replace all expressions of the form $(c + t)$ with $(t . c)$, when $c$ is a cube of $X$ and $t$ a configuration with values in $A_s(X)$.

  There is a bit more to say for (iii).  First, note that since $\rho$ and $\partial_k \rho'$ are continuous $\ell$-cocycles with small image, it is immediate that $\tilde{\rho}$ inherits these properties (see Remark \ref{rem:cocycle-defs}); the challenge is to show the invariance property.

    We have
    \[
      \tilde\rho([c, t.c']_k) = \tilde\rho([c, c']_k) + \tilde\rho([c', t.c']_k)
    \]
    and so it suffices to show that
    \[
      \rho([c', t.c']_k) = \rho'(c') - \rho'(t.c')
    \]
    for all $c'$ and $t \in T_{1}^\ell$. (Indeed, this implies that $\tilde\rho([c',t.c']_k)=0$.) Again this is very similar to Proposition \ref{prop:ds-cocycle-fixing}: we compute
    \begin{align*}
      \rho([c', t.c']_k) &= \int_{T_{1}^\ell} \left(\rho([c', w.c']_k) + \rho([w.c', t.c']_k) \right) d\mu_{T_{1}^\ell}(w) \\
                    &= \int_{T_{1}^\ell} \rho([c', w.c']_k) d\mu_{T_{1}^\ell}(w) - \int_{T_{1}^\ell} \rho([t.c', w.c']_k) d\mu_{T_{1}^\ell}(w) \\
                    &= \rho'(c') - \rho'(t.c')
    \end{align*}
    as required, once again after reparameterization of the last integral.
\end{proof}

\begin{proof}[Proof of Lemma \ref{lem:inductive-cocycle-special}]
We iterate Lemma \ref{lem:inductive-cocycle-more-special}.
We first write $\rho=\partial_1\rho'_1+\tilde \rho_1$, then in the $k$-th step we
apply Lemma \ref{lem:inductive-cocycle-more-special} to the cocycle
$\tilde \rho_{k-1}$ and write $\tilde\rho_{k-1}=\partial_k\rho'_k+\tilde\rho_k$.
Combining these equations we obtain the desired decomposition
\[
\rho=\sum_{k=1}^{\ell}\partial_k\rho'_k+\tilde\rho_\ell.
\]

It remains to show that $\tilde \rho_\ell$ descends to a cocycle on $\pi_{s-1}(X)$.
We first observe that $\tilde \rho_k$ is invariant under $[0,T_{1}^\ell]_k$ and this property
is inherited by $\tilde\rho_j$ for all $j>k$. Indeed, the cocycle $\tilde\rho_j$ is constructed
in the proof of Lemma  \ref{lem:inductive-cocycle-more-special} as the difference
of a $[0,T_{1}^\ell]_k$-invariant cocycle and an average of $[0,T_{1}^\ell]_k$-invariant cocycles.

Thus $\tilde\rho_\ell$ is invariant under $[0,T_{1}^\ell]_k$ for all $k$.
  Recalling that reflection simply negates the cocycle, i.e.~$\tilde\rho_\ell([c_0,c_1]_k) = -\tilde\rho_\ell([c_1,c_0]_k)$, it follows that $\tilde\rho_\ell$ is also invariant under $[T_1^\ell, 0]_k$.
It is clear that these groups generate $\cD_s(A_s(X))$, and hence $\tilde{\rho}_\ell$ is invariant under the
action of this group.

Now it follows that $\tilde\rho_\ell$ descends to a function on $C^{\ell}(\pi_{s-1}(X))$.  However, it is not completely obvious that this yields a cocycle. One needs to check that configurations $\tilde{c}_0, \tilde{c}_1, \tilde{c}_2$ on $\pi_{s-1}$ such that $[\tilde{c}_0, \tilde{c}_1]$ and $[\tilde{c}_1, \tilde{c}_2]$ are cubes of $\pi_{s-1}(X)$, \emph{lift} to corresponding configurations $c_0, c_1, c_2$ on $X$ in a compatible way.  Fortunately this is always true; see \cite{GMV1}*{Lemma 7.5}. 
\end{proof}

This completes the proof of Theorem \ref{thm:cocycle-special}.  The generalization to a proof of Theorem \ref{thm:cocycle-main} involves no new ideas at all, and only a small amount of further justification.

We will have to recall some notions from the \emph{relative} structure theory expounded in \cite{GMV1}*{Section 7}.  
What we shall need is that the canonical factor $X \to \pi_{s-1}(X)$ of a nilspace $X$ of degree $s$, has an analogue for fibrations.

\begin{proposition}
  \label{prop:relative-weak-structure}
  Let $\beta \colon X \to Y$ be a fibration of degree $s$ between compact ergodic cubespaces $X$ and $Y$ that obey the glueing axiom.  Then there is a map $\pi_{\beta,s-1} \colon X \to \pi_{\beta,s-1}(X)$ where
  \begin{itemize}
       \item the fibration $\beta$ factors as $X \xrightarrow[\pi_{\beta,s-1}]{} \pi_{\beta,s-1}(X) \xrightarrow[\tilde{\beta}]{} Y$;
 \item the map $\tilde{\beta}$ is a fibration of degree $(s-1)$;
    \item there is a compact abelian group $A_s(\beta)$, the \emph{structure group of the fibration}, which acts continuously and freely on all of $X$ and whose orbits are precisely the fibers of $\pi_{\beta,s-1}$;
    \item a similar statement holds for cubes; specifically, $C^k(\cD_s(A_s(\beta)))$ acts pointwise on $C^k(X)$, and its orbits are precisely the fibers of the map $\pi_{\beta,s-1} \colon C^k(X) \to C^k(\pi_{\beta,s-1}(X))$.
  \end{itemize}
\end{proposition}
For the details, see \cite{GMV1}*{Proposition 7.12 and Theorem 7.19}. 
(Apply both results with $(s+1)$ in place of $s$.)

We can now state the inductive step that provides a technical generalization of Lemma \ref{lem:inductive-cocycle-special}.

\begin{lemma}
  \label{lem:inductive-cocycle-general}
  Let $A$ be a compact abelian Lie group with the metric $d_A$ as above, and let $s \ge 1$, $\ell \ge 1$ be given.  Then there exists $\eps = \eps(s, \ell, A)$ such that the following holds.

  Suppose $X, Y, \beta, \rho$ are as in Theorem \ref{thm:cocycle-main}.  Then there is a decomposition
  \[
    \rho(c) = \sum_{k=1}^\ell\partial_k\rho'_k(c) + \tilde{\rho}(\pi_{\beta,s-1}(c))
  \]
  where
  \begin{itemize}
    \item $\rho'_k \colon C^{\ell-1}(X) \to A$ is an $(\ell-1)$-cocycle for each $k$, and
    \item $\tilde{\rho} \colon C^\ell(\pi_{\beta,s-1}(X)) \to A$ is an $\ell$-cocycle on $\pi_{\beta,s-1}(X)$,
  \end{itemize}
  such that both $\rho'$ and $\tilde{\rho}$ have small image in the usual sense.
\end{lemma}

The deduction of Theorem \ref{thm:cocycle-main} is again very similar.
\begin{proof}[Proof of Theorem \ref{thm:cocycle-main} assuming Lemma \ref{lem:inductive-cocycle-general}]
  The only thing that changes is the base case; the rest of the argument is the same as for Theorem \ref{thm:cocycle-special}.  If we have a fibration $\beta \colon X \to Y$ of degree $0$ between two ergodic cubespaces $X$ and $Y$, then $\beta$ is an isomorphism.  In particular, an $\ell$-cocycle on $X$ is the same thing (under $\beta$) as an $\ell$-cocycle on $Y$, so again Theorem \ref{thm:cocycle-main} is trivial in this case.
\end{proof}

Finally, the proof of Lemma \ref{lem:inductive-cocycle-general} is unchanged from that of Lemma \ref{lem:inductive-cocycle-special}, replacing all appearances of $\pi_{s-1}$ with $\pi_{\beta,s-1}$ and $A_s(X)$ with $A_s(\beta)$.  All appeals to the weak structure theory are legitimized by Proposition \ref{prop:relative-weak-structure}.

\subsection{The uniqueness result}
\label{sec:uniqueness}

We now prove Theorem \ref{thm:cocycle-uniqueness}.  Our proof is very closely modelled on that of \cite{CS12}*{Lemma 3.25}.

We work up to the result in several incremental stages, each of which is in fact a special case of the theorem.  The first is classical.

\begin{lemma}
  \label{lem:homs-discrete}
  Suppose $A$ is a compact abelian Lie group equipped with the metric $d_A$.  Then there exists $\eps = \eps(A) > 0$ such that the following holds.

  If $G$ is an abelian group and $\phi \colon G \to A$ is a group homomorphism such that $d_A(0, \phi(g)) \le \eps$ for all $g \in G$, then $\phi$ is the trivial homomorphism.
\end{lemma}
\begin{proof}
  This is equivalent to saying that $A$ has ``no small subgroups'': there is some neighbourhood of the identity in $A$ containing no non-trivial subgroups.

  Since we know $A$ is isomorphic to $(\RR/\ZZ)^d \times K$ for some finite $K$, we may argue this directly.  Let
  \[
    \pi \colon \RR^d \times K \to (\RR/\ZZ)^d \times K
  \]
  denote the projection map.  Choose some $\delta > 0$ such that $\pi$ identifies $B_\delta(0) \subseteq (\RR/\ZZ)^d \times K$ bijectively with with the corresponding open ball $B_\delta(0) \times \{0\}$ in $\RR^d \times K$.

  Let $x \in B_{\delta/2}(0) \setminus \{0\} \subseteq (\RR/\ZZ)^d \times K$, and let $\tilde{x}$ be the corresponding lift to $\RR^d$.  Then there is some positive integer $n$ such that $\delta/2 \le |n \tilde{x}| < \delta$.  It follows that $n x \in B_\delta(0) \setminus B_{\delta/2}(0) \subseteq (\RR/\ZZ)^d \times K$, and therefore $x$ is not contained in any subgroup of $B_{\delta/2}(0)$.
\end{proof}

Note that if we interpret $G$ as a nilspace of degree $1$ (i.e.~considering $\cD_1(G)$) then any homomorphism $\phi \colon G \to A$ has $\partial^2 \phi = 0$, so this is a special case of Theorem \ref{thm:cocycle-uniqueness}.

Having shown that the space of \emph{linear} maps $G \to A$ is discrete, we bootstrap this to a statement about \emph{polynomial} maps.

\begin{lemma}
  \label{lem:uniqueness-d1-case}
  Suppose $A$ is a compact abelian Lie group equipped with the metric $d_A$.  Fix $\ell \ge 1$.  Then there exists $\eps = \eps(\ell, A) > 0$ such that the following holds.

  Let $G$ be any abelian group, and consider the cubespace structure $\cD_1(G)$ on $G$. Let $\gamma \colon G \to A$ be a map such that $\partial^\ell \gamma \equiv 0$ and $d_A(\gamma(g), \gamma(g')) \le \eps$ for all $g, g' \in G$.

  Then $\gamma$ is constant.
\end{lemma}
\begin{proof}
  We proceed by induction on $\ell$.  The case $\ell = 1$ is trivial: since $[x,y] \in C^1(G)$ for   every $x,y \in G$ we have $\partial^1 \gamma([x,y]) = \gamma(x) - \gamma(y) = 0$ for all $x,y$, and hence $\gamma$ is constant.

  Now take $\ell > 1$.  Note that an $\ell$-cube in $\cD_1(G)$ has the form $[c, t+c]$ for some $t \in G$, where $t+c$ denotes the pointwise shift $\omega \mapsto t + c(\omega)$ of $c$.  Defining the ``derivative''
  \[
    \gamma_t(x) = \gamma(x) - \gamma(x+t)
  \]
  we note that
  \[
    \partial^\ell \gamma([c, t+c]) = \partial^{\ell-1} \gamma_t(c)  \ .
  \]
  By hypothesis, for all $t$ we have $\partial^{\ell-1} \gamma_t \equiv 0$ and hence, for an appropriate choice of $\eps$, by induction $\gamma_t \equiv \alpha(t)$ is a constant function for all $t$.

  But now observe that
  \begin{align*}
    \alpha(t) + \alpha(t') &= \gamma_t(0) + \gamma_{t'}(t) \\
                 &= \gamma(0) - \gamma(t) + \gamma(t) - \gamma(t+t') \\
                 &= \alpha(t+t')
  \end{align*}
  and trivially $\alpha(0) = 0$, $\alpha(-t) = -\alpha(t)$, so $\alpha \colon G \to A$ is a group homomorphism.  Moreover, by hypothesis $d_A(0, \alpha(t)) \le \eps$ for all $t$; so by Lemma \ref{lem:homs-discrete} $\alpha \equiv 0$.  Hence, $\gamma(x) = \gamma(0) - \alpha(x) = \gamma(0)$ is constant.
\end{proof}

Once again it is clear that this is a special case of the general result.  The point is that we can use the weak structure theory to decompose our nilspace $X$ of degree $s$ into a tower of extensions
\[
  X  \to \pi_{s-1}(X) \to \dots \to \pi_{0}(X) = \{\ast\}
\]
whose fibers are compact abelian groups $A_k(X)$ equipped with a nilspace structure $\cD_k(A_k)$.  By applying Lemma \ref{lem:uniqueness-d1-case} on the fibers one at a time, we deduce the result for $X$.

\begin{proof}[{Proof of Theorem \ref{thm:cocycle-uniqueness}}]
  We proceed by induction on $s$.  The case $s=0$ is trivial, since an ergodic nilspace of degree $0$ is just the $1$-point space $\{\ast\}$.

  Suppose $s>0$. Fix $y \in \pi_{s-1}(X)$ and consider the fiber $\pi_{s-1}^{-1}(y)$.  Clearly the restriction of $\gamma$ to $\pi_{s-1}^{-1}(y)$ still has the same properties.  By the weak structure theorem (\cite{GMV1}*{Theorem 5.4}), 
  $\pi_{s-1}^{-1}(y)$ is isomorphic as a cubespace to $\cD_s(A_s(X))$, and hence $\gamma|_{\pi_{s-1}^{-1}(y)}$ can be identified with a map $\gamma_y \colon A_s(X) \to A$,\footnote{On this occasion $\gamma_y$ does \emph{not} denote a derivative.} such that
  \[
    \partial^\ell \gamma_y \colon C^\ell(\cD_s(A_s(X))) \to A
  \]
  is identically zero.

  Note, however, that $C^\ell(\cD_1(A_s(X)))$ is contained in $C^\ell(\cD_s(A_s(X)))$.  Hence, \emph{a fortiori} $\gamma$ satisfies the hypotheses of Lemma \ref{lem:uniqueness-d1-case}, and so (for a suitable choice of $\eps$) is constant.

  Since $y$ was arbitrary, we deduce that $\gamma$ is constant on fibers of $\pi_{s-1}$.  Hence, $\gamma$ factors as a map $\gamma' \colon \pi_{s-1}(X) \to A$ on a nilspace of degree at most $(s-1)$, and it is clear that $\gamma'$ inherits all the properties of $\gamma$.  By inductive hypothesis, $\gamma'$ is constant and hence $\gamma$ is constant.
\end{proof}

\appendix
\section{A technicality}

We record here a fact that would distract from the discussion if recorded in-place.

The following states that a free action by a compact group is ``robustly free'', in the sense that only small elements come close to stabilizing points.
\begin{proposition}
  \label{prop:robust-free-action}
  Suppose $G$ is a compact metric group acting continuously and freely on a compact metric space $X$.  For all $\eps > 0$ there exists $\delta > 0$ such that the following holds: if $x \in X$, $g \in G$ and $d_X(x, g.x) \le \delta$ then $d_G(\id, g) \le \eps$.
\end{proposition}
\begin{proof}
  Suppose for contradiction that there exists $\eps > 0$ and sequences $x_i \in X$, $g_i \in G$ such that $d_G(\id, g_i) \ge \eps$ for all $i$ but $d_X(x_i, g_i . x_i) \rightarrow 0$.  Passing to convergent subsequence, we may assume $x_i \to x$ and $g_i \to g$ for some $x \in X$ and $g \in G$.  Now, $d_G(\id, g) \ge \eps$ and in particular $g \ne \id$.  By continuity of the action map $G \times X \to X$, we have that $g_i . x_i \to g . x$, and so $d_X(x, g.x) = 0$, contradicting freeness of the action.
\end{proof}

\bibliography{paper-2}
\bibliographystyle{halpha}

\end{document}